\documentclass[12pt]{article}

\usepackage{hyperref}

\usepackage{amsmath,amsthm,amsfonts,amscd,amssymb,bm,mathrsfs}
\usepackage{enumerate}
\usepackage{graphicx,xcolor,epstopdf,subfigure}
\usepackage{multirow,cases}
\usepackage{authblk}
\usepackage{geometry}
\usepackage[title]{appendix}
\DeclareMathOperator*{\argmin}{arg\,min}

\newcommand{\B}{{\mathbb{B}}}

\newcommand{\R}{{\mathbb{R}}}

\newcommand{\I}{{\mathbb{I}}}
\newcommand{\Su}{{\mathcal{S}}}
\newcommand{\loss}{{\mathcal{L}}}
\newcommand{\regu}{{\mathcal{P}}}
\newcommand{\obj}{\Psi}

\newcommand{\T}{{\mathcal{T}}}
\newcommand{\Aa}{{\mathbb{A}}}
\newcommand{\Bb}{{\mathbb{B}}}
\newcommand{\X}{{\mathscr{X}}}
\newcommand{\N}{{\mathcal{N}}}
\newcommand{\XX}{{\mathfrak{X}}}
\newcommand{\Z}{{\mathfrak{Z}}}

\newcommand{\Q}{{\mathcal{Q}}}
\newcommand{\normm}[1]{{\left\vert\kern-0.25ex\left\vert\kern-0.25ex\left\vert #1
\right\vert\kern-0.25ex\right\vert\kern-0.25ex\right\vert}}
\newcommand{\inm}[2]{\langle\langle #1,#2 \rangle\rangle}

\newtheorem{Assumption}{Assumption}
\newtheorem{Lemma}{Lemma}
\newtheorem{Definition}{Definition}
\newtheorem{Proposition}{Proposition}
\newtheorem{Remark}{Remark}
\newtheorem{Theorem}{Theorem}

\title{Low-rank matrix estimation via nonconvex spectral regularized methods in errors-in-variables matrix regression}
\author[1]{Xin Li}
\author[2]{Dongya Wu}
\affil[1]{School of Mathematics, Northwest University, Xi’an, 710069, P. R. China}
\affil[2]{School of Information Science and Technology, Northwest University, Xi’an, 710069, P. R. China}

\date{}

\begin{document}
\maketitle

\begin{abstract}
High-dimensional matrix regression has been studied in various aspects, such as statistical properties, computational efficiency and application to specific instances including multivariate regression, system identification and matrix compressed sensing. Current studies mainly consider the idealized case that the covariate matrix is obtained without noise, while the more realistic scenario that the covariates may always be corrupted with noise or missing data has received little attention. We consider the general errors-in-variables matrix regression model and proposed a unified framework for low-rank estimation based on nonconvex spectral regularization. Then in the statistical aspect, recovery bounds for any stationary points are provided to achieve statistical consistency. In the computational aspect, the proximal gradient method is applied to solve the nonconvex optimization problem and is proved to converge in polynomial time. Consequences for specific  matrix compressed sensing models with additive noise and missing data are obtained via verifying corresponding regularity conditions. Finally, the performance of the proposed nonconvex estimation method is illustrated by numerical experiments.
\end{abstract}


\section{Introduction}\label{sec-intro}

Matrix regression model, as one of the most popular and important model in the filed of signal processing and machine learning, has received extensive attention and gained widespread success in recent decades \cite{wainwright2014structured, zhou2014regularized}. In order to alleviate the high-dimensional challenge that the dimension of the underlying matrix parameter $d_1\times d_2$ far exceeds the sample size $N$, some low-dimensional structural constraints must be imposed on the parameter space to achieve statistical estimation consistency. The most commonly-used constraint is the rank constraint, which refers to that the true matrix parameter is of rank $r$ much less than the matrix dimension $d_1\times d_2$. However, unfortunately, due to the nonsmoothness and nonconvexity of the rank function, minimizing the rank of a matrix is NP-hard \cite{Natarajan1995SparseAS}. Therefore, researchers seek to find other surrogates to estimate a low-rank matrix, such as the nuclear norm, the Schatten-$p$ norm, the max norm and the row/column sparse functions; see \cite{recht2010guaranteed,negahban2011estimation,candes2010power,zhou2014regularized} and references therein. Among these surrogates, the nuclear norm, analogous to the $\ell_1$ norm as a convex relaxation of the cardinality for a vector in linear regression, is the convex envelope of the rank function under unit norm constraint \cite{Fazel2001ARM}, and thus is widely used for low-rank approximation. In the statistical aspect, there have been fruitful strong guarantees that the nuclear norm regularization method often generate a solution near the true parameter with lower rank, even a minimum rank solution in special cases \cite{recht2010guaranteed,negahban2011estimation,candes2010power}. In the computational aspect, even though the nuclear norm is nonsmooth, the resulting optimization problem is still computationally favorable and can be solved efficiently with fast convergence rates using modern optimization algorithms such as the proximal gradient algorithm \cite{agarwal2012fast, Mazumder2010SpectralRA}.

Despite the success, the nuclear norm regularization method also inherited the deficiency of the $\ell_1$ norm penalized method in linear regression that larger and more informative singular values are penalized to induce significant estimation bias. On the other hand, extensive studies have shown that nonconvex regularizers such as the smoothly clipped absolute deviation penalty (SCAD) \cite{fan2001variable} and the minimax concave penalty (MCP) \cite{Zhang2010} eliminate the estimation bias to a degree and enjoy more refined statistical convergence rates in linear regression. In light of the advantages, researchers have begun to adapt the nonconvex regularizers to matrix regression models, that is to impose nonconvex regularizers on the singular values of a matrix. For example, authors in \cite{zhou2014regularized} proposed a class of regularized matrix regression methods based on spectral regularization that works for for a variety of penalization functions, including the Lasso \cite{tibshirani1996regression}, Elastic net \cite{zou2005regularization}, SCAD and MCP. They also developed a highly efficient and scalable estimation algorithm and established nonasymptotic convergence rates. \cite{Lu2014GeneralizedNN} proposed an iteratively reweighted nuclear norm algorithm to solve the nonconvex nonsmooth low-rank minimization problem and proved that any limit point is a stationary point. 
A unified framework was presented in \cite{gui2015towards} low-rank matrix estimation with nonconvex regularizers. The nonconvex estimator is solved by a proximal gradient homotopy algorithm and is shown to enjoy a faster statistical convergence rate than that of the convex nuclear norm regularized estimator. In order to avoid the expensive full SVD in the proximal step, authors in \cite{yao2015fast} and \cite{yao2017large} proposed to automatically threshold the singular values obtained from the proximal operator via the power method and achieved a fast algorithmic convergence rate. 

The results above are based on the idealized hypothesis that the collected covariates are clean enough, while noisy response variables are always considered to be noisy in statistical models. However, in many real-world applications, due to economical or instrumental constraints, the collected covariates may usually be perturbed and thus tend to be noisy or have missing values. This is called the errors-in-variables model. There have been fruitful researches on estimation in errors-in-variables models under low dimensional situation; see, e.g., \cite{bickel1987efficient,carroll2006measurement} and references therein. Estimation bias is generally corrected by imputation or using maximum likelihood estimation methods such as the expectation maximization (EM) algorithm in classical low-dimensional framework. Nevertheless, the method of imputation may change original data and inaccurate estimates of real data may introduce new noise, leading to more complicated problems. The EM algorithm may face the challenge of a huge amount of computation and a slow convergence rate. Moreover, due to the potential nonconvexity of objective functions in errors-in-variables regression, the EM algorithm may terminate in undesirable local optima. Under high-dimensional scenarios, authors in \cite{sorensen2015measurement} have pointed out that one can only get misleading results when the method for clean data is applied naively to the noisy data. Meanwhile, due to the huge amount of data, measurement errors may increase exponentially, especially for the case of missing data. This phenomenon makes the limitations of classical methods more prominent in face of high-dimensional data. A more practical method is to perform error-corrected estimation on the promise of keeping the original data information unchanged as much as possible.

Recently, many researchers have considered estimation in high-dimensional errors-in-variables regression using the original data. In linear errors-in-variables models, for example, Loh and Wainwright proposed nonconvex regularized estimator via substituting the unobserved matrices involved in the ordinary least squares loss function with unbiased surrogates, and established statistical errors for global and stationary solutions \cite{loh2012high,loh2015regularized}. To overcome the nonconvexity, Datta and Zou defined the nearest positive semi-definite matrix and developed the convex conditional Lasso (CoCoLasso) which enjoys the benefits of convex optimization and possesses nice estimation accuracy simultaneously \cite{datta2017cocolasso}. Rosenbaum and Tsybakov proposed a modified form of the Dantzig selector \cite{candes2007the}, called matrix uncertainty selector (MUS) for variable selection \cite{rosenbaum2010sparse}. Further development of MUS included modification to achieve statistical consistency, and generalization to deal with the cases of unbounded and dependent measurement errors as well as generalized linear models \cite{belloni2016an,belloni2017linear,rosenbaum2013improved,sorensen2018covariate}. Li et al. investigated a general nonconvex regularized $M$-estimator, which can be applied to errors-in-variables sparse linear regression, and analysed the statistical and computational properties \cite{li2020sparse}. Li et al. proposed a corrected decorrelated score test and a score type estimator \cite{li2021inference}. Other methods fall out the category of regularization methods. For instance, \cite{chen2013noisy} modified the orthogonal matching pursuit algorithm for variable selection in errors-in-variables linear regression; the measurement error boosting (MEBoost) algorithm \cite{brown2019meboost} was based on the idea of classical estimation equation and implemented error-corrected variable selection at every iterative path. 

In multi-response errors-in-variables models, Wu et al. developed a new methodology called convex conditioned sequential sparse learning (COSS) that combines the power of the sequential sparse factor regression and the nearest positive semi-definite matrix projection, thus possesses the benefits of stepwise scalability and convexity in large-scale association analyses \cite{wu2020scalable}. Li et al. proposed a nonconvex error-corrected estimator and established the statistical consistency and algorithmic linear convergence rate for global solutions of the estimator based on nuclear norm regularization \cite{li2023Lowrank}. 

However, there is still little attention paying to the more general and widely-used matrix regression model under the errors-in-variables framework. The matrix regression model, as a generic and unified observation model, includes many different models such as the multi-response model, matrix compressed sensing, matrix completion and so on. Our work here aims to deal with the estimation problem in errors-in-variables matrix regression. A unified estimation framework based on nonconvex spectral regularization is proposed to reduce the bias due to measurement errors. Statistical and computational guarantees are then established. The key to ensuring the success of low-rank recovery relies on suitable regularity conditions on the nonconvex loss function and the regularizer, while conditions used in linear regression cannot be used to analyse matrix regression \cite{Cands2009ThePO}. 



To date, however, an open question is whether or not appropriate conditions holds for errors-in-variables matrix regression. In handling this question, the contributions of this paper are as follows. First, we propose a specialized form of restricted strong convexity (RSC) and restricted strong smoothness (RSM) on the nonconvex loss function and proved that these conditions hold for errors-in-variables matrix regression with overwhelming probability, by virtue of nontrivial matrix analysis on concentration inequalities and random matrix theory (cf. Propositions \ref{prop-add}--\ref{prop-mis}). Second, in the statistical aspect, we provide
recovery bounds for any stationary point of the nonconvex optimization problem to obtain statistical consistency (cf. Theorem \ref{thm-stat}). This recovery bound does not rely on any specific algorithms, and thus any numerical procedure can recover the true low-rank matrix as long as it converges to a stationary point. Last but not least, in the statistical aspect, we modified the proximal gradient method \cite{nesterov2007gradient} to solve the nonconvex optimization problem and proved the linear convergence to a global solution (cf. Theorem \ref{thm-algo}).


The remainder of this article is organized as follows. In Sect. \ref{sec-prob}, we propose a general nonconvex  estimator based on spectral regularization. Some regularity conditions are imposed on the loss function and the regularizer facilitate the analysis. In Sect. \ref{sec-main}, we establish our main results on statistical recovery bounds and computational convergence rates. In Sect. \ref{sec-conse}, probabilistic consequences on the regularity conditions for specific errors-in-variables models are obtained. In Sect. \ref{sec-num}, several numerical experiments are performed to demonstrate theoretical results. Conclusions and future work are discussed in Sect. \ref{sec-con}. Technical proofs are deferred to the Supplementary Material.

We end this section by introducing useful notations. For a vector $\beta\in \R^d$ and an index set $J\subseteq \{1,2,\dots,d\}$, we use $\beta_J$ to denote the vector in which $(\beta_J)_i=\beta_i$ for $i\in J$ and zero elsewhere, $|J|$ to denote the cardinality of $J$, and $J^c=\{1,2,\dots,d\}\setminus J$ to denote the complement of $J$. For $d\geq 1$, let $\I_d$ stand for the $d\times d$ identity matrix. For a matrix $X\in \R^{d_1\times d_2}$, let $X_{ij}\ (i=1,\dots,d_1,j=1,2,\cdots,d_2)$ denote its $ij$-th entry, $X_{i\cdot}\ (i=1,\dots,d_1)$ denote its $i$-th row, $X_{\cdot j}\ (j=1,2,\cdots,d_2)$ denote its $j$-th column, and vec$(X)\in \R^{d_1d_2}$ to denote its vectorized form. When $X$ is a square matrix, i.e., $d_1=d_2$, we use diag$(X)$ stand for the diagonal matrix with its diagonal elements equal to $X_{11},X_{22},\cdots,X_{d_1d_1}$. We write $\lambda_{\text{min}}(X)$ and $\lambda_{\text{max}}(X)$ to denote the minimal and maximum eigenvalues of a matrix $X$, respectively. For a matrix $\Theta\in \R^{d_1\times d_2}$, define $d=\min\{d_1,d_2\}$, and denote its singular values in decreasing order by $\sigma_1(\Theta)\geq \sigma_2(\Theta)\geq \cdots \sigma_d(\Theta)\geq 0$. We use $\normm{\cdot}$ to denote different types of matrix norms based on singular values, including the nuclear norm $\normm{\Theta}_*=\sum_{j=1}^{d}\sigma_j(\Theta)$, the spectral or operator norm $\normm{\Theta}_{\text{op}}=\sigma_1(\Theta)$, and the Frobenius norm $\normm{\Theta}_\text{F}=\sqrt{\text{trace}(\Theta^\top\Theta)}=\sqrt{\sum_{j=1}^{d}\sigma_j^2(\Theta)}$. All vectors are column vectors following classical mathematical convention. For a pair of matrices $\Theta$ and $\Gamma$ with equal dimensions, we let $\inm{\Theta}{\Gamma}=\text{trace}(\Theta^\top \Gamma)$ denote the trace inner product on matrix space. For a function $f:\R^d\to \R$, $\nabla f$ is used to denote the gradient when $f$ is differentiable, and $\partial f$ is used to denote the subdifferential that consists of all subgradients when $f$ is nondifferentiable but convex.

\section{Problem setup}\label{sec-prob}

In this section, we first propose a general nonconvex error-corrected estimator based on nonconvex spectral regularization. Then some regularity conditions to guarantee the statistical and computational properties are also given in detail.

\subsection{General nonconvex estimator}

In this paper, we are mainly interested in the high-dimensional estimation scenario where the number of unknown matrix elements $d_1\times d_2$ can be possibly much larger than the number of observations $N$. From the theoretical aspect, it has been already known that one cannot achieve consistent estimation under this high-dimensional setting unless that the model space is endowed with additional structures, such as low-rankness in matrix estimation problems. Empirical evidence has also demonstrated that low-rank matrices always arise in real applications, such as collaborative filtering \cite{sagan2021lowrank} and multi-task learning \cite{wu2019joint}. 
In the following, we shall impose the low-rank constraint on the parameter space. For a matrix $\Theta\in \R^{d_1\times d_2}$, let $d=\min\{d_1,d_2\}$ and we use $\sigma(\Theta)$ to represent the vector which is formed by the singular values of $\Theta$ in decreasing order, that is, $\sigma(\Theta)=(\sigma_1(\Theta),\sigma_2(\Theta),\cdots,\sigma_d(\Theta))\top$ and $\sigma_1(\Theta)\geq \sigma_2(\Theta)\geq \cdots \geq \sigma_r(\Theta)\geq \sigma_{r+1}(\Theta)\geq \cdots \sigma_{d}(\Theta)$. Specifically, the true parameter $\Theta^*\in \R^{d_1\times d_2}$ is assumed to be of either exact low-rank, that is, the rank of $\Theta^*$ is far less than $d_1\times d_2$, or near low-rank, which refers to the case that $\Theta^*$ can be well approximated by an exact low-rank matrix. The matrix $\ell_q$-ball is used to measure the degree of low-rank, which is defined as follows, for $q\in [0,1]$, and a radius $R_q>0$,
\begin{equation}\label{eq-lqball}
\B_q(R_q):=\{\Theta\in \R^{d_1\times d_2}\big|\sum_{i=1}^{d}|\sigma_i(\Theta)|^q\leq R_q\}.
\end{equation}
It is worth noting that these balls are not true balls when $q\in[0,1)$ due to the nonconvexity. When $q=0$, the matrix $\ell_0$-ball refers to the exact low-rank assumption, i.e., the rank of a matrix is at most $R_0$; while when $q\in (0,1]$, the matrix $\ell_q$-ball refers to the near low-rank assumption that enforces some decay rate on the ordered singular values of the matrix $\Theta\in \B_q(R_q)$. In the following, unless otherwise specified, we assume that the true parameter $\Theta^*\in \B_q(R_q)$ for a fixed value $q\in [0,1]$.

Let us then cast the matrix estimation problem into the framework of the regularized $M$-estimation problem. For a random variable $B:\mathcal{S} \to \mathcal{B}$ defined on the probability space $(\mathcal{S},\mathcal{F},\mathbb{P})$, with $\mathbb{P}$ belonging to a parameterized set $\mathcal{P}=\{\mathbb{P}_\Theta\big| \Theta\in \Omega \subseteq \R^{d_1\times d_2}\}$. Assume that one collects a sample of $n$ independent and identically distributed (i.i.d.) observations of this random variable $B$, written as $B_1^N:=(B_1,B_2,\cdots,B_N)$. The goal is then to estimate a true parameter $\Theta^*\in \Omega$ such that the observed data $B_1^N$ is generated by a true probability distribution $\mathbb{P}_{\Theta^*}\in \mathcal{P}$. To this end, we introduce a loss function $\loss_N:\R^{d_1\times d_2}\times \mathcal{B}^N\to \R_+$, with its value $\loss_N(\Theta;B_1^N)$ measuring the fitness between any parameter $\Theta\in \Omega$ and the collected data $B_1^N\in \mathcal{B}^N$. Recalling the low-rank constraint imposed on the parameter space, we shall consider the following regularized $M$-estimator
\begin{equation}\label{eq-esti}
\hat{\Theta} \in \argmin_{\Theta\in \Omega\subseteq \R^{d_1\times d_2}}\{\loss_N(\Theta;B_1^N)+\regu_\lambda(\Theta)\},
\end{equation}
where $\lambda>0$ is a regularization parameter providing a tradeoff between data fitness and low-rankness, and $\regu_\lambda:\R^{d_1\times d_2}\to \R$ is a regularizer depending on $\lambda$ imposing sparsity on the singular values of the estimator $\hat{\Theta}$, and thus low-rankness of $\hat{\Theta}$. 

The regularizer is written as $\regu_\lambda(\Theta)=\sum_{j=1}^{d}p_\lambda(\sigma_j(\Theta))$, with $p_\lambda:\R \to \R$ being a univariate function of the singular values of a matrix $\Theta$. Furthermore, the univariate function $p_\lambda(\cdot)$ is assumed to be decomposed as $p_\lambda(\cdot)=q_\lambda(\cdot)+\lambda|\cdot|$, where $q_\lambda(\cdot)$ is a concave component and $|\cdot|$ is the absolute value function. Therefore, one sees that the regularizer can be decomposed as $\regu_\lambda(\cdot)=\Q_\lambda(\cdot)+\lambda\normm{\cdot}_*$, where $\Q_\lambda(\cdot)$ is the concave component given by 
\begin{equation}\label{eq-qlambda}
\Q_\lambda(\cdot)=\sum_{j=1}^{d}q_\lambda(\sigma_j(\cdot)),
\end{equation} 
and $\normm{\cdot}_*$ is the nuclear norm function. 

In the following analysis, the loss function $\loss_N(\cdot;B_1^n)$ is not required to be convex, but only differentiable. The regularizer $\regu_\lambda$ can also be nonconvex. Due to the potential nonconvexity, the feasible region is specialized as a convex set as follows
\begin{equation}\label{eq-feasi}
\Omega:=\{\Theta\in \R^{d_1\times d_2}\big| \normm{\Theta}_*\leq \omega\},
\end{equation}
where $\omega>0$ must be chosen to ensure the feasibility of $\Theta^*$, i.e., $\Theta^*\in \Omega$. Any matrix $\Theta\in \Omega$ will also satisfy the side constraint $\|\Theta\|_*\leq \omega$. Then the existence of global solutions is guaranteed by Weierstrass extreme value theorem providing the continuity of the loss function $\loss_N(\cdot;B_1^n)$ and the regularizer $\regu_\lambda$. Hereinafter to simplify the notation, the shorthand $\loss_N(\cdot)$ for $\loss_N(\cdot;B_1^N)$ will be adopted. 

\subsection{Regularity conditions}

We first impose some regularity conditions on the empirical loss function $\loss_N$. Regularity conditions named RSC/RSM have been introduced to analyse linear/matrix regression to control the statistical error and guarantee nice algorithmic performance, and is applicable when the loss function is nonquadratic or nonconvex; see, e.g.,\cite{agarwal2012fast,loh2015regularized,negahban2011estimation}. When there exists no measurement error, researchers have shown that the RSC/RSM conditions are satisfied by a variety class of random matrices with high probability \cite{agarwal2012fast,negahban2011estimation}. 

However, it is still unknown whether or not a suitable form of RSC/RSM exists for errors-in-variables matrix regression. In this paper, we provide a positive answer for this question by proposing the following general RSC/RSM conditions. Verification for specific measurement error models involves probabilistic discussions under high-dimensional scaling and will be given in Section \ref{sec-conse}.

We begin with defining the first-order Taylor series expansion around a matrix $\Theta'$ in the direction of $\Theta$ as
\begin{equation}\label{taylor}
\T(\Theta,\Theta'):=\loss_N(\Theta)-\loss_N(\Theta')-\inm{\nabla\loss_N(\Theta')}{\Theta-\Theta'}.
\end{equation} 

The RSC condition takes two types of forms, one is used to control statistical errors for any stationary point; the other one is used for the analysis of algorithmic convergence rates together with the RSM condition. See Definitions \ref{asup-sta-rsc}-\ref{asup-rsm} below.


\begin{Definition}\label{asup-sta-rsc}
The function $\loss_N$ is said to satisfy the statistical restricted strong convexity with parameters $\alpha_1>0$ and $\tau_1>0$ if
\begin{equation}\label{eq-sta-rsc}
\inm{\nabla\loss_N(\Theta^*+\Delta)-\nabla\loss_N(\Theta^*)}{\Delta}\geq \alpha_1\normm{\Delta}_\text{F}^2-\tau_1\normm{\Delta}_*^2,\quad \forall\  \Delta\in \R^{d_1\times d_2}.
\end{equation}
\end{Definition}

\begin{Definition}\label{asup-alg-rsc}
The function $\loss_N$ is said to satisfy the algorithmic restricted strong convexity with parameters $\alpha_2>0$ and $\tau_2>0$ if
\begin{equation}\label{eq-alg-rsc}
\T(\Theta,\Theta')\geq \alpha_2\normm{\Theta-\Theta'}_\text{F}^2-\tau_2\normm{\Theta-\Theta'}_*^2,\quad \forall\  \Theta,\Theta'\in \R^{d_1\times d_2}.
\end{equation}
\end{Definition}

\begin{Definition}\label{asup-rsm}
The function $\loss_N$ is said to satisfy the restricted strong smoothness with parameters $\alpha_3>0$ and $\tau_3>0$ if
\begin{equation}\label{eq-alg-rsm}
\T(\Theta,\Theta')\leq \alpha_3\normm{\Theta-\Theta'}_\text{F}^2+\tau_3\normm{\Theta-\Theta'}_*^2,\quad \forall\ \Theta,\Theta'\in R^{d_1\times d_2}.
\end{equation}
\end{Definition}


Then several regularity conditions are imposed on the nonconvex regularizer $\regu_\lambda(\cdot)$ in terms of the univariate functions $p_\lambda(\cdot)$ and $q_\lambda(\cdot)$.

\begin{Assumption}\mbox{}\par\label{asup-regu}
\begin{enumerate}[\rm(i)]
\item $p_\lambda$ satisfies $p_\lambda(0)=0$ and is symmetric around zero, that is, $p_\lambda(t)=p_\lambda(-t)$ for all $t\in \R$.
\item For $t>0$, the function $t\mapsto \frac{p_\lambda(t)}{t}$ is nonincreasing in $t$;
\item $p_\lambda$ is differentiable for all $t\neq 0$ and subdifferentiable at $t=0$, with $\lim\limits_{t\to 0^+}p'_\lambda(t)=\lambda$.
\item On the nonnegative real line, $p_\lambda$ is nondecreasing and concave.
\item For $t>t'$, there exists a positive constant $\mu\geq 0$ such that
    \begin{equation}\label{cond-qlambda}
    q'_\lambda(t)-q'_\lambda(t')\geq -\mu(t-t').
    \end{equation}
\end{enumerate}
\end{Assumption}
Note that condition (ii) implies that on the nonnegative line, the function $p_\lambda$ is subadditive. It is easy to check that the nuclear norm satisfies Assumption \ref{asup-regu}. Other nonconvex regularizers such as SCAD and MCP are also contained in our framework. Precisely, fixing $a>2$ and $b>0$, the function $p_\lambda$ for the SCAD regularizer is defined as 
\begin{equation*}
p_\lambda(t):=\left\{
\begin{array}{l}
\lambda|t|,\ \ \text{if}\ \  |t|\leq \lambda,\\
-\frac{t^2-2a\lambda|t|+\lambda^2}{2(a-1)},\ \  \text{if}\ \ \lambda<|t|\leq a\lambda,\\
\frac{(a+1)\lambda^2}{2},\ \ \text{if}\ \ |t|>a\lambda,
\end{array}
\right.
\end{equation*}
and the function $p_\lambda$ for the MCP regularizer is defined as
\begin{equation*}
p_\lambda(t):=\left\{
\begin{array}{l}
\lambda|t|-\frac{t^2}{2b},\ \ \text{if}\ \  |t|\leq b\lambda,\\
\frac{b\lambda^2}{2},\ \ \text{if}\ \ |t|>b\lambda.
\end{array}
\right.
\end{equation*}
Then 
\begin{equation}\label{SCAD-q-2}
q_\lambda(t)=\left\{
\begin{array}{l}
0,\ \ \text{if}\ \  |t|\leq \lambda,\\
-\frac{t^2-2\lambda|t|+\lambda^2}{(2(a-1)},\ \  \text{if}\ \ \lambda<|t|\leq a\lambda,\\
\frac{(a+1)\lambda^2}{2}-\lambda|t|,\ \ \text{if}\ \ |t|>a\lambda,
\end{array}
\right.
\end{equation}
for SCAD with $\mu=\frac{1}{a-1}$, and
\begin{equation}\label{MCP-q-2}
q_\lambda(t)=\left\{
\begin{array}{l}
-\frac{t^2}{2b},\ \ \text{if}\ \  |t|\leq b\lambda,\\
\frac{b\lambda^2}{2}-\lambda|t|,\ \  \text{if}\ \ |t|> b\lambda,
\end{array}
\right.
\end{equation}
for MCP with $\mu=\frac{1}{b}$, respectively, for condition (v). 

At the end of this section, we provide three technical lemmas, which tell us some general properties of the nonconvex regularizer $\regu_\lambda$ and the concave component $\Q_\lambda$. The first lemma is from \cite[Theorem 1]{Rotfeld1967RemarksOT} and \cite[Theorem 1]{Yue2016API} concerning about inequalities of matrix singular values with the proof omitted.

\begin{Lemma}\label{lem-tri}
Let $\Theta,\Theta'\in \R^{d_1\times d_2}$ be two given matrices and $d=\min\{d_1,d_2\}$. Let $f:\R_+\to \R_+$ be a concave increasing function satisfying $f(0)=0$. Then we have that
\begin{align}
\sum_{j=1}^{d}f(\sigma_j(\Theta+\Theta'))&\leq \sum_{j=1}^{d}f(\sigma_j(\Theta))+\sum_{j=1}^{d}f(\sigma_j(\Theta')),\label{eq-tri+}\\
\sum_{j=1}^{d}f(\sigma_j(\Theta-\Theta'))&\geq \sum_{j=1}^{d}f(\sigma_j(\Theta))-\sum_{j=1}^{d}f(\sigma_j(\Theta')).\label{eq-tri-}
\end{align}
\end{Lemma}

\begin{Lemma}\label{lem-regu}
Suppose that $\regu_\lambda$ satisfy Assumption \ref{asup-regu}. Then for any $\Theta,\Theta'\in \R^{d_1\times d_2}$, one has that
\begin{align}
\regu_\lambda(\Theta+\Theta')&\leq \regu_\lambda(\Theta)+\regu_\lambda(\Theta'),\label{eq-regu+}\\
\regu_\lambda(\Theta-\Theta')&\geq \regu_\lambda(\Theta)-\regu_\lambda(\Theta').\label{eq-regu-}
\end{align}
\end{Lemma}
\begin{proof}
Since the singular values of a matrix is always nonnegative, the univariate function $p_\lambda$ actually  satisfies $p_\lambda: \R_+\to \R_+$. Then by Assumption \ref{asup-regu} (i) and (iv), Lemma \ref{lem-tri} is applicable to concluding that 
\begin{align*}
\sum_{j=1}^{d}p_\lambda(\sigma_j(\Theta+\Theta'))&\leq \sum_{j=1}^{d}p_\lambda(\sigma_j(\Theta))+\sum_{j=1}^{d}p_\lambda(\sigma_j(\Theta')),\\
\sum_{j=1}^{d}p_\lambda(\sigma_j(\Theta-\Theta'))&\geq \sum_{j=1}^{d}p_\lambda(\sigma_j(\Theta))-\sum_{j=1}^{d}p_\lambda(\sigma_j(\Theta')).
\end{align*}
The conclusion then follows directly from the definition of the regularizer $\regu_\lambda$.
\end{proof}

\begin{Lemma}\label{lem-qlambda}
Let $\Q_\lambda$ be defined in \eqref{eq-qlambda}. Then for any $\Theta,\Theta'\in \R^{d_1\times d_2}$, the following relations are true:
\begin{subequations}
\begin{align}
&\inm{\nabla\Q_\lambda(\Theta)-\nabla\Q_\lambda(\Theta')}{\Theta-\Theta'} \geq -\mu\normm{\Theta-\Theta'}_\text{F}^2,\label{lem-qlambda-11}\\
&\inm{\nabla\Q_\lambda(\Theta)-\nabla\Q_\lambda(\Theta')}{\Theta-\Theta'}\leq 0, \label{lem-qlambda-12}\\
&\Q_\lambda(\Theta)\geq \Q_\lambda(\Theta')+\inm{\nabla\Q_\lambda(\Theta')}{\Theta-\Theta'}- \frac{\mu}{2}\normm{\Theta-\Theta'}_\text{F}^2,\label{lem-qlambda-13}\\
&\Q_\lambda(\Theta)\leq \Q_\lambda(\Theta')+\inm{\nabla\Q_\lambda(\Theta')}{\Theta-\Theta'}.\label{lem-qlambda-14}
\end{align}
\end{subequations}
\end{Lemma}
\begin{proof}
For arbitrary matrices $\Theta,\Theta'\in \R^{d_1\times d_2}$, let $\sigma(\Theta),\sigma(\Theta')$ be the vectors of singular values of $\Theta,\Theta'$ in decreasing order, respectively, and $d=\min\{d_1,d_2\}$. For the sake of simplicity, we use $\sigma,\sigma'$ to denote $\sigma(\Theta),\sigma(\Theta')$, respectively. Then we have the singular value decompositions for $\Theta,\Theta'$ as follows:
\begin{equation*}
\begin{aligned}
\Theta&=UDV^\top,\\
\Theta'&=U'D'{V'}^\top,
\end{aligned}
\end{equation*}
where $D,D'\in \R^{d\times d}$ are diagonal with $D=\text{diag}(\sigma), D'=\text{diag}(\sigma')$. By Assumption \ref{asup-regu}(iv)-(v), one has for each pair of singular values of $\Theta,\Theta'$: ($\sigma_j,\sigma'_j$), $j=1,2,\cdots,d$, it holds that 
\begin{equation*}
-\mu(\sigma_j-\sigma'_j)^2\leq (q'_\lambda(\sigma_j)-q'_\lambda(\sigma'_j))(\sigma_j-\sigma'_j)\leq 0,
\end{equation*}
Then by the definitions of $D,D'$, one has that 
\begin{equation*}
-\mu\normm{\Theta-\Theta'}_\text{F}^2\leq \inm{\nabla\Q_\lambda(UDV^\top)-\nabla\Q_\lambda(U'D'{V'}^\top)}{\Theta-\Theta'}\leq 0,
\end{equation*}
from which \eqref{lem-qlambda-11} and \eqref{lem-qlambda-12} follows directly.
Then by \cite[Theorem 2.1.5 and Theorem 2.1.9]{nesterov2013introductory}, it follows from \eqref{lem-qlambda-11} and \eqref{lem-qlambda-12} that
the convex function $-\Q_\lambda$ satisfies
\begin{align*}
-\Q_\lambda(\Theta)&\leq -\Q_\lambda(\Theta')+\inm{\nabla(-\Q_\lambda(\Theta'))}{\Theta-\Theta'}+ \frac{\mu}{2}\normm{\Theta-\Theta'}_\text{F}^2,\\
-\Q_\lambda(\Theta)&\geq -\Q_\lambda(\Theta')+\inm{\nabla(-\Q_\lambda(\Theta'))}{\Theta-\Theta'},
\end{align*}
which respectively implies that the function $\Q_\lambda$ satisfies \eqref{lem-qlambda-13} and \eqref{lem-qlambda-14}.
The proof is complete.
\end{proof}

\section{Main results}\label{sec-main}

In this section, the main results on statistical guarantee on the recovery bound for the general nonconvex estimator \eqref{eq-esti} and computational guarantee on convergence rates for the proximal gradient algorithm. These results are deterministic in nature. Stochastic consequences for the errors-in-variables matrix regression model will be discussed in the next section. 

Before proceeding, we need some additional notations to facilitate the following analysis. First let $\obj(\Theta)=\loss_N(\Theta)+\regu_\lambda(\Theta)$ denote the objective function to be minimized. Recall that the regularizer can be decomposed as $\regu_\lambda(\Theta)=\Q_\lambda(\Theta)+\lambda\normm{\Theta}_*$, then it follows that $\obj(\Theta)=\loss_N(\Theta)+\Q_\lambda(\Theta)+\lambda\normm{\Theta}_*$. Donote $\bar{\loss}_N(\Theta)=\loss_N(\Theta)+\Q_\lambda(\Theta)$, then we have that $\obj(\Theta)=\bar{\loss}_N(\Theta)+\lambda\normm{\Theta}_*$. Through this decomposition, it is easy to see that the objective function is decomposed into a differentiable but possibly nonconvex function and a nonsmooth but convex function.

Note that the parameter matrix $\Theta^*$ has a singular value
decomposition of the form $\Theta^*=U^*D^*{V^*}^\top$, where $U^*\in \R^{d_1\times d}$ and $V^*\in \R^{d_2\times d}$
are orthonormal matrices with $d=\min\{d_1,d_2\}$, and without loss of generality, assume that $D$ is diagonal with singular values in nonincreasing order, i.e., $\sigma_1(\Theta^*)\geq \sigma_2(\Theta^*)\geq \cdots \sigma_d(\Theta^*)\geq 0$. For each integer $r\in \{1, 2,\cdots,d\}$, We use $U^r\in \R^{d_1\times r}$ and $V^r\in \R^{d_2\times r}$ to denote the sub-matrices consisting of left and right singular vectors indexed by the top $r$ largest
singular values of $\Theta^*$, respectively. Then we define the following two subspaces of $\R^{d_1\times d_2}$ associated with $\Theta^*$ as:
\begin{subequations}\label{eq-sub}
\begin{align}
\Aa(U^r,V^r)&:=\{\Delta\in \R^{d_1\times d_2}\big|\text{row}(\Delta)\subseteq \text{col}(V^r), \text{col}(\Delta)\subseteq \text{col}(U^r)\},\label{eq-sub1}\\
\Bb(U^r,V^r)&:=\{\Delta\in \R^{d_1\times d_2}\big|\text{row}(\Delta)\perp \text{col}(V^r), \text{col}(\Delta)\perp \text{col}(U^r)\},\label{eq-sub2}
\end{align}
\end{subequations}
where $\text{row}(\Delta)\in \R^{d_2}$ and $\text{col}(\Delta)\in \R^{d_1}$ respectively represent the row space and column space of the matrix $\Delta$. When the sub-matrices $(U^r,V^r)$ are explicit from the context, we use the shorthand notation $\Aa^r$ and $\Bb^r$ instead. Definitions of $\Aa^r$ and $\Bb^r$ have been introduced in \cite{agarwal2012fast,negahban2011estimation} to study low-rank estimation problems without measurement errors, in order to show the decomposability of the nuclear norm, that is, $\normm{\Theta+\Theta'}_*=\normm{\Theta}_*+\normm{\Theta'}_*$ holds for any arbitrary pair of matrices $\Theta\in \Aa^r$ and $\Theta'\in \Bb^r$, indicating that the nuclear norm is decomposable with respect to the subspaces $\Aa^r$ and $\Bb^r$. 

Still consider the singular value decomposition $\Theta^*=U^*D^*{V^*}^\top$.
For any positive number $\eta>0$ to be chosen, 
a set corresponding to $\Theta^*$ is defined as following:
\begin{equation}\label{eq-thresh}
K_\eta:=\{j\in\{1,2,\cdots,d\}\big||\sigma_j(\Theta^*)|\geq \eta\}.
\end{equation}
Using the above notations, one sees that the matrix $U^{|K_\eta|}$ (resp., $V^{|K_\eta|}$) represents the $d_1\times |K_\eta|$ (resp., the $d_2\times |K_\eta|$) orthogonal sub-matrix comprising of the singular vectors corresponding to the first $|K_\eta|$ singular values of $\Theta^*$.
Recall the subspace defined in \eqref{eq-sub2}, and define the matrix 
\begin{equation}\label{eq-resi}
\Theta^*_{K_\eta^c}:=\Pi_{\Bb^{|K_\eta|}}(\Theta^*).
\end{equation}
Then it is obvious to see that the matrix $\Theta^*_{K_\eta^c}$ is of rank at most $d-|K_\eta|$ and has singular values $\{\sigma_j(\Theta^*
)\big| j\in K_\eta^c\}$. Moreover, since the true parameter is assumed to be of near low-rank, i.e., $\Theta^*\in \B_q(R_q)$ (cf. \eqref{eq-lqball}), the cardinality of set $K_\eta$ (cf. \eqref{eq-thresh}) and the approximation error in the nuclear norm ( i.e., $\normm{\Theta^*_{K_\eta^c}}_*$) can both be bounded from above. In fact, it follows immediately from standard derivations (see, e.g., \cite{negahban2011estimation}) that
\begin{equation}\label{eq-thresh-bound}
|K_\eta|\leq \eta^{-q}R_q\quad \mbox{and}\quad \normm{\Theta^*_{K_\eta^c}}_*\leq\eta^{1-q}R_q.
\end{equation}
Essentially, the cardinality of $K_\eta$ serves as the effective rank under the near low-rank assumption, when $\eta$ is suitably chosen. This fact has been discussed in \cite{Li2024LowrankME,negahban2011estimation}, and we shall also clarified it in the proof of Theorem \ref{thm-stat}; see Remark \ref{rmk-stat} (iv).

Consider an arbitrary matrix $\Delta\in \R^{d_1\times d_2}$ with a singular value decomposition $\Delta=UDV^\top$, where $U\in \R^{d_1\times d}$ and $V^*\in \R^{d_2\times d}$ are orthonormal matrices with $d=\min\{d_1,d_2\}$. Let $\sigma(\Delta)$ be the vector formed by the singular values of $\Delta$ in decreasing order. Let $r$ be a positive integer and $J=\{1,2,\cdots,2r\}$. 
Define 
\begin{equation}\label{eq-thetaj}
\Delta_J=U\text{diag}(\sigma_J(\Delta))V^\top\quad  \mbox{and}\quad \Delta_{J^c}=U\text{diag}(\sigma_{J^c}(\Delta))V^\top.
\end{equation}
Then it is easy to see that $\Delta=\Delta_J+\Delta_{J^c}$ and $\Delta_J\perp \Delta_{J^c}$.

With these notations, we now state a useful technical lemma that shows, for the true parameter matrix $\Theta^*$ and any matrix $\Theta$, we can decompose the error matrix $\Delta:=\Theta-\Theta^*$ as the sum of two matrices $\beta$ and $\beta'$ such that the
rank of $\beta$ is bounded. In addition, the difference between the regularizer imposed on $\Theta^*$ and $\Theta$ can be bounded from above in terms of the nuclear norms. 


\begin{Lemma}\label{lem-decom}
Let $\Theta\in \R^{d_1\times d_2}$ be an arbitrary matrix. Define the error matrix $\Delta:=\Theta-\Theta^*$ with a singular value decomposition as $\Delta=UDV^\top$. Let $r$ be a positive integer and $J=\{1,2,\cdots,2r\}$.
Let $\Delta_J$ and $\Delta_{J^c}$ be given in \eqref{eq-thetaj}. Then the following conclusions hold:\\
\rm (i) there exists a decomposition $\Delta=\Delta'+\Delta''$ such that the matrix $\Delta'$ with $\text{rank}(\Delta')\leq 2r$;\\
\rm (ii) $\regu_\lambda(\Theta^*)-\regu_\lambda(\Theta)\leq \lambda(\normm{\Delta_J}_*-\normm{\Delta_{J^c}}_*)+2\lambda\sum_{j=r+1}^d\sigma_j(\Theta^*)$.
\end{Lemma}
\begin{proof}
(i) The first part of this lemma is proved in  \cite[Lemma 3.4]{recht2010guaranteed}, and we provide the proof here for completeness. Write the singular value decomposition of $\Theta^*$ as $\Theta^*=U^*D^*{V^*}^\top$, where $U^*\in \R^{d_1\times d_1}$ and $V^*\in \R^{d_2\times d_2}$ are orthogonal matrices, and $D^*\in\R^{d_1\times d_2}$ is the matrix consisting of the singular values of $\Theta^*$. Define the matrix $\Xi={U^*}^\top\Delta V^*\in \R^{d_1\times d_2}$, and partition $\Xi$ in block form as follows
      \[
      \Xi:=\begin{pmatrix}
        \Xi_{11} & \Xi_{12} \\
        \Xi_{21} & \Xi_{22}
        \end{pmatrix},\ \text{where}\ \Xi_{11}\in \R^{r\times r}\ \text{and}\  \Xi_{22}\in \R^{(m_1-r)\times (m_2-r)}.
      \]
      Set the matrices as
      \[
      \Delta':=U^*\begin{pmatrix}
        0 & 0 \\
        0 & \Xi_{22}
        \end{pmatrix}{V^*}^\top\ \text{and}\  \Delta'':=\Delta-\Delta'.
      \]
      Then the rank of $\Delta'$ is upper bounded as
      \begin{equation*}
        \text{rank}(\Delta')=\text{rank}\begin{pmatrix}
        \Xi_{11} & \Xi_{12} \\
        \Xi_{21} & 0
        \end{pmatrix}\leq
        \text{rank}\begin{pmatrix}
        \Xi_{11} & \Xi_{12} \\
        0 & 0
        \end{pmatrix}+
        \text{rank}\begin{pmatrix}
        \Xi_{11} & 0 \\
        \Xi_{21} & 0
        \end{pmatrix}\leq 2r,
      \end{equation*}
      which established Lemma \ref{lem-decom}(i).\\
(ii) 
It follows from the constructions of $\Delta'$ and $\Delta''$
that $\sigma(\Delta')+\sigma(\Delta'')=\sigma(\Delta)$ with $\text{supp}(\sigma(\Delta'))\leq 2r$ and $\inm{\Delta'}{\Delta''}=0$.
    Note that the decomposition $\Theta^*=\Pi_{\Aa^r}(\Theta^*)+\Pi_{\Bb^r}(\Theta^*)$ holds. This equality, together with Lemma \ref{lem-regu}, implies that
      \begin{equation}\label{eq-lem4-2}
      \begin{aligned}
       \regu_\lambda(\Theta)&=\regu_\lambda[(\Pi_{\Aa^r}(\Theta^*)+\Delta'')+(\Pi_{\Bb^r}(\Theta^*)+\Delta')]\\
       &=\regu_\lambda[(\Pi_{\Aa^r}(\Theta^*)+\Delta'')-(-\Pi_{\Bb^r}(\Theta^*)-\Delta')]\\
       &\geq \regu_\lambda(\Pi_{\Aa^r}(\Theta^*)+\Delta'')-\regu_\lambda(-\Pi_{\Bb^r}(\Theta^*)-\Delta')\\
       &\geq \regu_\lambda(\Pi_{\Aa^r}(\Theta^*))+\regu_\lambda(\Delta'')-\regu_\lambda(\Pi_{\Bb^r}(\Theta^*))-\regu_\lambda(\Delta').
      \end{aligned}
      \end{equation}
      Consequently, we have
      \begin{equation}\label{eq-lem4-3}
      \begin{aligned}
      \regu_\lambda(\Theta^*)-\regu_\lambda(\Theta)
      &\leq
      \regu_\lambda(\Theta^*)-\regu_\lambda(\Pi_{\Aa^r}(\Theta^*))-\regu_\lambda(\Delta'')+\regu_\lambda(\Pi_{\Bb^r}(\Theta^*))+\regu_\lambda(\Delta')\\
      &\leq
      \regu_\lambda(\Delta')-\regu_\lambda(\Delta'')+2\regu_\lambda(\Pi_{\Bb^r}(\Theta^*))\\
      &\leq
      \regu_\lambda(\Delta_J)-\regu_\lambda(\Delta_{J^c})+2\regu_\lambda(\Pi_{\Bb^r}(\Theta^*)),
      \end{aligned}
      \end{equation}
      where the last inequality is from the definitioen of the set $J$ and Assumption \ref{asup-regu} (iv).
Then it follows from \cite[Lemma 5 and Lemma 6]{loh2013local} that 
$$\sum_{i=1}^dp_\lambda(\sigma_i(\Delta_J))-\sum_{i=1}^dp_\lambda(\sigma_i(\Delta_{J^c}))\leq \lambda\left(\sum_{i=1}^d\sigma_i(\Delta_J)-\sum_{i=1}^d\sigma_j(\Delta_{J^c})\right),$$
$$\sum_{i=1}^dp_\lambda(\sigma_i(\Pi_{\Bb^r}(\Theta^*)))\leq \lambda\sum_{i=1}^d\sigma_i(\Pi_{\Bb^r}(\Theta^*)).$$
Combining these two inequalities with \eqref{eq-lem4-3} and the definition of $\regu_\lambda$, we arrive at that $\regu_\lambda(\Theta^*)-\regu_\lambda(\Theta)\leq \lambda(\normm{\Delta_J}_*-\normm{\Delta_{J^c}}_*)+2\lambda\normm{\Pi_{\Bb^r}(\Theta^*)}_*$. Then Lemma \ref{lem-decom} (ii) follows from the fact that $\normm{\Pi_{\Bb^r}(\Theta^*)}_*=\sum_{j=r+1}^{d}\sigma_j(\Theta^*)$. The proof is complete.
\end{proof}

\subsection{Statistical recovery bounds}

Recall the feasible region $\Omega$ given in \eqref{eq-feasi}. We shall provide the recovery bound for each stationary point $\tilde{\Theta}\in \Omega$ of the nonconvex optimization problem \eqref{eq-esti} satisfying the first-order necessary condition:
\begin{equation}\label{1st-cond}
\inm{\nabla\loss_N(\tilde{\Theta})+\nabla\regu_\lambda(\tilde{\Theta})}{\Theta-\tilde{\Theta}}\geq 0,\quad \text{for all}\ \Theta\in \Omega.
\end{equation}

\begin{Theorem}\label{thm-stat}
Let $R_q>0$ and $\omega>0$ be positive numbers such that $\Theta^*\in \B_q(R_q)\cap \Omega$. Let $\tilde{\Theta}$ be a stationary point of the optimization problem \eqref{eq-esti}. Suppose that the nonconvex regularizer $\regu_\lambda$ satisfies Assumption \ref{asup-regu}, and that the empirical loss function $\loss_N$ satisfies the \emph{RSC} condition \eqref{eq-sta-rsc} with $\alpha_1>\mu$.
Assume that $(\lambda, \omega)$ are chosen to satisfy
\begin{equation}\label{eq-lambda-sta}
\lambda\geq 2\max\{\normm{\nabla\loss_N(\Theta^*)}_\emph{op},4\omega\tau_1\},
\end{equation}
then we have that 
\begin{equation}\label{eq-l2bound}
\normm{\hat{\Theta}-\Theta^*}_\text{F}^2\leq 9R_q\left(\frac{\lambda}{\alpha_1-\mu}\right)^{2-q},
\end{equation}
\begin{equation}\label{eq-l1bound}
\normm{\hat{\Theta}-\Theta^*}_*\leq (24\sqrt{2}+8)R_q\left(\frac{\lambda}{\alpha_1-\mu}\right)^{1-q}.
\end{equation}
\end{Theorem}
\begin{proof}
Set $\tilde{\Delta}:=\tilde{\Theta}-\Theta^*$. By the RSC condition \eqref{eq-sta-rsc}, one has that
\begin{equation}\label{eq-lem3-1}
\inm{\nabla\loss_N(\tilde{\Theta})-\nabla\loss_N(\Theta^*)}{\tilde{\Delta}}\geq \alpha_1\normm{\tilde{\Delta}}_\text{F}^2-\tau_1\normm{\tilde{\Delta}}_*^2.
\end{equation}
On the other hand, it follows from \eqref{lem-qlambda-11} and  \eqref{lem-qlambda-14} in Lemma \ref{lem-qlambda} that
\begin{equation*}
\begin{aligned}
\inm{\nabla\regu_\lambda(\tilde{\Theta})}{\Theta^*-\tilde{\Theta}}
&= \inm{\nabla\Q_\lambda(\tilde{\Theta})+\lambda\tilde{G}}{\Theta^*-\tilde{\Theta}}\\
&\leq \inm{\nabla\Q_\lambda(\Theta^*)}{\Theta^*-\tilde{\Theta}}+\mu\normm{\Theta^*-\tilde{\Theta}}_\text{F}^2+\inm{\lambda\tilde{G}}{\Theta^*-\tilde{\Theta}}\\
&\leq \Q_\lambda(\Theta^*)-\Q_\lambda(\tilde{\Theta})+\mu\normm{\Theta^*-\tilde{\Theta}}_\text{F}^2+\inm{\lambda\tilde{G}}{\Theta^*-\tilde{\Theta}},
\end{aligned}
\end{equation*}
where $\tilde{G}\in \partial \normm{\tilde{\Theta}}_*$.
Moreover, since the function $\normm{\cdot}_*$ is convex, one has that
\begin{equation*}
\normm{\Theta^*}_*-\normm{\tilde{\Theta}}_*\geq \inm{\tilde{G}}{\Theta^*-\tilde{\Theta}}.
\end{equation*}
This, together with the former inequality, implies that
\begin{equation}\label{eq-lem3-2}
\inm{\nabla\regu_\lambda(\tilde{\Theta})}{\Theta^*-\tilde{\Theta}}\leq 
\regu_\lambda(\Theta^*)-\regu_\lambda(\tilde{\Theta})+\mu\normm{\tilde{\Delta}}_\text{F}^2.
\end{equation}
Then combining \eqref{eq-lem3-1}, \eqref{eq-lem3-2} and \eqref{1st-cond} (with $\Theta^*$ in place of $\Theta$), we have that
\begin{equation}\label{eq-lem3-3}
\begin{aligned}
\alpha_1\normm{\tilde{\Delta}}_\text{F}^2-\tau_1\normm{\tilde{\Delta}}_*^2
&\leq -\inm{\nabla\loss_N(\Theta^*)}{\tilde{\Delta}}+\regu_\lambda(\Theta^*)-\regu_\lambda(\tilde{\Theta})+\mu\normm{\tilde{\Delta}}_\text{F}^2\\
&\leq  \normm{\nabla\loss_N(\Theta^*)}_\text{op}\normm{\tilde{\Delta}}_*+\regu_\lambda(\Theta^*)-\regu_\lambda(\tilde{\Theta})+\mu\normm{\tilde{\Delta}}_\text{F}^2\\
&\leq  \frac{\lambda}{2}\normm{\tilde{\Delta}}_*+\regu_\lambda(\Theta^*)-\regu_\lambda(\tilde{\Theta})+\mu\normm{\tilde{\Delta}}_\text{F}^2\\
\end{aligned}
\end{equation}
where the second inequality is from H{\"o}lder's inequality, and the last inequality is from assumption \eqref{eq-lambda-sta}.
Let $J$ denote the index set corresponding to the $2r$ largest singular values of $\tilde{\Delta}$ with $r$ to be chosen later and recall the definitions given in \eqref{eq-thetaj}. It then follows from Lemma \ref{lem-decom} (ii) and noting the fact that $\normm{\tilde{\Delta}}_*\leq \normm{\Theta^*}_*+\normm{\tilde{\Theta}}_*\leq 2\omega$, one has from \eqref{eq-lem3-3} that 
\begin{equation}\label{eq-lem3-4}
\begin{aligned}
(\alpha_1-\mu)\normm{\tilde{\Delta}}_\text{F}^2&\leq \frac{\lambda}{2}\normm{\tilde{\Delta}}_*+2\omega\tau_1\normm{\tilde{\Delta}}+\lambda(\normm{\tilde{\Delta}_J}_*-\normm{\tilde{\Delta}_{J^c}}_*)+2\lambda\sum_{j=r+1}^d\sigma_j(\Theta^*)\\
&\leq \frac{3}{4}\lambda\normm{\tilde{\Delta}}_*+\lambda(\normm{\tilde{\Delta}_J}_*-\normm{\tilde{\Delta}_{J^c}}_*)+2\lambda\sum_{j=r+1}^d\sigma_j(\Theta^*)\\
&\leq \frac{3}{4}\lambda(\normm{\tilde{\Delta}_J}_*+\normm{\tilde{\Delta}_{J^c}}_*)+\lambda(\normm{\tilde{\Delta}_J}_*-\normm{\tilde{\Delta}_{J^c}}_*)+2\lambda\sum_{j=r+1}^d\sigma_j(\Theta^*),
\end{aligned}
\end{equation}
where the second inequality is due to assumption \eqref{eq-lambda-sta}
and the last inequality is from triangle inequality. Since $\alpha_1>\mu$ by assumption, one has by the former inequality that \begin{equation}\label{eq-thm1-cone}
\normm{\tilde{\Delta}_{J^c}}_*\leq 7\normm{\tilde{\Delta}_J}_*+8\sum_{j=r+1}^d\sigma_j(\Theta^*).
\end{equation}
Combining \eqref{eq-lem3-4} and \eqref{eq-thm1-cone} yields that
\begin{equation*}
\begin{aligned}
(\alpha_1-\mu)\normm{\tilde{\Delta}}_\text{F}^2&\leq \frac{7}{4}\lambda\normm{\tilde{\Delta}_J}_*-\frac{1}{4}\normm{\tilde{\Delta}_{J^c}}_*+2\lambda\sum_{j=r+1}^d\sigma_j(\Theta^*)\\
&\leq \frac{7}{4}\lambda\sqrt{2r}\normm{\tilde{\Delta}}_\text{F}+2\lambda\sum_{j=r+1}^d\sigma_j(\Theta^*),
\end{aligned}
\end{equation*}
where the second inequality is due to the fact that $\text{rank}(\tilde{\Delta}_J)\leq 2r$. Then it follows that
\begin{equation}\label{eq-thm1-rbound}
\normm{\hat{\Delta}}_\text{F}^2\leq \frac{98r\lambda^2+32(\alpha_1-\mu)\lambda\sum_{j=r+1}^d\sigma_j(\Theta^*)}{16(\alpha_1-\mu)^2}.
\end{equation}
Recall the set $K_\eta$ defined in \eqref{eq-thresh} and set $r=|K_\eta|$. Combining \eqref{eq-thm1-rbound} with \eqref{eq-thresh-bound} and setting $\eta=\frac{\lambda}{\alpha_1-\mu}$, we arrive at \eqref{eq-l2bound}. Moreover, it follows from \eqref{eq-thm1-cone} that  
\begin{equation*}
\normm{\tilde{\Delta}}_*\leq 8\normm{\tilde{\Delta}_J}_*+8\sum_{j=r+1}^d\sigma_j(\Theta^*)\leq 8\sqrt{2r}\normm{\tilde{\Delta}}_\text{F}+8\sum_{j=r+1}^d\sigma_j(\Theta^*),
\end{equation*}
and thus \eqref{eq-l1bound} holds. The proof is complete.
\end{proof}
\begin{Remark}\label{rmk-stat}
{\rm (i)} There are two parameters involved in the nonconvex optimization problem \eqref{eq-esti}, i.e., the regularization parameter $\lambda$ and the radius of the side constraint $\omega$. On the synthetic data, since the true parameter $\Theta^*$ is settled beforehand, these two parameters can be determined as we did in Section \ref{sec-num}. In real data analysis where $\Theta^*$ is always unknown, one might think that the side constraint $\normm{\Theta^*}_*\leq \omega$ is restrictive. Even so, Theorem \ref{thm-stat} still provides some heuristic to find the scale for the radius $\omega$. In detail, it follows from assumption \eqref{eq-lambda-sta} that the relation $\lambda\geq 8\tau\omega$ holds, based on which methods such as cross-validation can be adopted to tune these two parameters. The choice on $\lambda$ and $\omega$ in Theorem \ref{thm-algo} can also be decided in this way.

{\rm (ii)} Note the quantity $\alpha_1-\mu$ appearing in the denominators of the recovery bounds in Theorem \ref{thm-stat}, which actually plays the role of balancing the nonconvexity of the estimator \eqref{eq-esti} (Theorem \ref{thm-algo} also involves a similar quantity and can be explained in a similar way).  Specifically, $\alpha_1$ measures the degree of curvature of the empirical loss function $\loss_N$ and $\mu$ measures the degree of nonconvexity of the penalty $\regu_\lambda$, a fact indicating that these two quantities play opposite roles in the estimation procedure. Larger values of $\mu$ result in more severe nonconvexity of the low-rank regularizer and thus a worse behavior of the objective function \eqref{eq-esti}, while larger values of $\alpha_1$ means more curvature of the loss function and thus leading to a better estimation. Hence the requirement that $\alpha_1>\mu$ is used to control this oppositional relationship and guarantee a good performance of local optima.

{\rm (iii)} Theorem \ref{thm-stat} provides a unified framework for low-rank matrix estimation in a generic observation model via nonconvex optimization and demonstrates that the recovery bound on the squared Frobenius norm for all the stationary points of the nonconvex estimator \eqref{eq-esti} scale as $\normm{\hat{\Theta}-\Theta^*}_\text{F}^2=O(\lambda^{2-q}R_q)$, which covers the specific multi-response regression with nuclear norm regularization in \cite[Theorem 1]{Li2024LowrankME}. Combining assumption \eqref{eq-lambda-sta} and probabilistic discussions on specific error models in the next section (cf. Propositions \ref{prop-add} and \ref{prop-mis}), one sees that it is suitable to choose the regularization parameter as $\lambda=\Omega\left(\sqrt{d_1d_2}\frac{\log d_1+\log d_2}{N}\right)$. Providing $R_q\left(\sqrt{d_1d_2}\frac{\log d_1+\log d_2}{N}\right)^{1-q/2}=o(1)$, this recovery bound implies the statistical consistency of the estimator $\tilde{\Theta}$. Furthermore, due to the nonconvexity of the optimization problem \eqref{eq-esti}, it is always difficult to obtain a global solution. Nevertheless, this issue is not significant in our framework since we have established recovery bounds for any stationary points. Theorem \ref{thm-stat} thus is independent of any specific algorithms, suggesting that any numerical procedure can consistently recover the true low-rank matrix provided that it can converge to a stationary point.

{\rm (vi)} Theorem \ref{thm-stat} provides the recovery bound under the more general near low-rank assumption that $\Theta^*\in \B_q(R_q)$ with $q\in [0,1]$. Since the low-rankness of $\Theta^*$ is measured via the matrix $\ell_q$-norm and larger values of $q$ means higher rank, \eqref{eq-l2bound} indicates that the convergence rate slows down as $q$ increases to 1. In addition, Theorem \ref{thm-stat} also sheds some insight on the effective rank of a near low-rank matrix $\Theta^*$. Actually, when the regularization parameter $\lambda$ is chosen as $\lambda=\Omega\left(\sqrt{d_1d_2}\frac{\log d_1+\log d_2}{N}\right)$, and the threshold $\eta$ in \eqref{eq-thresh} is set to $\eta=\frac{\lambda}{\alpha_1-\mu}$ as we did in the proof of Theorem \ref{thm-stat}, the cardinality of the set $K_\eta$ (cf. \eqref{eq-thresh}) acts as the effective rank under the near low-rank assumption. This special value is used to provide a balance between the estimation and approximation errors for a near low-rank matrix and has also been discussed in \cite{Li2024LowrankME,negahban2011estimation}. When the sample size $N$ increases, this effective rank also increases since $\lambda$ tends to $0$ (cf. \eqref{eq-thresh-bound}), a fact indicating that with more samples collected, it is likely to recover more smaller singular values of a near low-rank matrix. 

{\rm (v)} \cite{gui2015towards} also considered the problem of low-rank matrix estimation with nonconvex penalty. However, the loss function there is still assumed to be convex and only recovery bounds for global solutions are established. \cite{Li2024LowrankME} studied the multi-response errors-in-variables regression and proposed a nonconvex estimator based on a nonconvex loss function and a nuclear norm regularizer. The results there are only applicable for a global solution whereas Theorem \ref{thm-stat} is a much stronger result holding for any stationary point and covering a more general class of nonconvex regularizers beyond the nuclear norm.

\end{Remark}

\subsection{Computational convergence rates}

The proximal gradient method \cite{nesterov2007gradient} is now applied to solve the nonconvex optimization problem \eqref{eq-esti} through a simple modification of the objective function and is proved to converge geometrically under the RSC/RSM conditions. Recall that the regularizer can be decomposed as $\regu_\lambda(\cdot)=\Q_\lambda(\cdot)+\lambda\normm{\cdot}_*$, and define the modified loss function $\bar{\loss}_N(\cdot)=\loss_N(\cdot)+\Q_\lambda(\cdot)$. Then the optimization objective function can be written as $\obj(\cdot)=\bar{\loss}_N(\cdot)+\lambda\normm{\cdot}_*$. It is easy to see that the optimization objective function is decomposed into a differentiable but nonconvex function and a nonsmooth but convex function (i.e., the nuclear norm). 

Recall the feasible region $\Su=\{\Theta\in \R^{d_1\times d_2}\big|\normm{\Theta}_*\leq \omega\}$ given in \eqref{eq-feasi}. Applying the proximal gradient method proposed in \cite{nesterov2007gradient} to \eqref{eq-esti}, we obtain a sequence of iteration points $\{\Theta^t\}_{t=0}^\infty$ as
\begin{equation}\label{algo-pga}
\Theta^{t+1}\in \argmin_{\Theta\in \Su}\left\{\frac{1}{2}\normm{\Theta-\left(\Theta^t-\frac{\nabla{\bar{\loss}_N(\Theta^t)}}{v}\right)}_\text{F}^2+\frac{\lambda}{v}\normm{\Theta}_*\right\},
\end{equation}
where $\frac{1}{v}$ is the step size.

Given $\Theta^t$, one can follow \cite{loh2015regularized} to generate the next iteration point $\Theta^{t+1}$ via the following three steps; see \cite[Appendix C.1]{loh2015regularized} for details.
\begin{enumerate}[(1)]
\item First optimize the unconstrained optimization problem
\begin{equation*}
\hat{\Theta^t}\in \argmin_{\Theta\in \R^{d_1\times d_2}}\left\{\frac{1}{2}\normm{\Theta-\left(\Theta^t-\frac{\nabla{\bar{\loss}_N(\Theta^t)}}{v}\right)}_\text{F}^2+\frac{\lambda}{v}\normm{\Theta}_*\right\}.
\end{equation*}
\item If $\normm{\hat{\Theta^t}}_*\leq \omega$, set $\Theta^{t+1}=\hat{\Theta^t}$.
\item Otherwise, if $\normm{\hat{\Theta^t}}_*>\omega$, optimize the constrained optimization problem
\begin{equation*}
\Theta^{t+1}\in \argmin_{\Theta\in \Su}\left\{\frac{1}{2}\normm{\Theta-\left(\Theta^t-\frac{\nabla{\bar{\loss}_N(\Theta^t)}}{v}\right)}_\text{F}^2\right\}.
\end{equation*}
\end{enumerate}

Before we give our main computational result that the sequence generated by \eqref{algo-pga} converges geometrically to a small neighborhood of any global solution $\hat{\Theta}$, some notations are needed to simplify the expositions. Define the Taylor error $\bar{\T}(\Theta,\Theta')$ for the modified loss function $\bar{\loss}_n$ as
\begin{equation}\label{taylor-modified}
\bar{\T}(\Theta,\Theta')=\T(\Theta,\Theta')+\Q_\lambda(\Theta)-\Q_\lambda(\Theta')-\inm{\nabla\Q_\lambda(\Theta')}{\Theta-\Theta'}.
\end{equation}
Recall the RSC and RSM conditions in \eqref{eq-alg-rsc} and \eqref{eq-alg-rsm}, respectively. Throughout this section, we set 
$\tau:=\max\{\tau_2,\tau_3\}$. Recall the true underlying parameter $\Theta^*\in \B_q(R_q)$ (cf. \eqref{eq-lqball}). Let $\hat{\Theta}$ be a global solution of the optimization problem \eqref{eq-esti}. Then unless otherwise specified, we define
\begin{align}
&\bar{\epsilon}_{\text{stat}}:=8\lambda^{-\frac{q}{2}}R_q^{\frac12}\left(\sqrt{2}\normm{\hat{\Theta}-\Theta^*}_\text{F}+\lambda^{1-\frac{q}2}R_q^{\frac12}\right),\label{bar-epsilon}\\
&\kappa:= \left\{1-\frac{2\alpha-\mu}{8v}+\frac{512\tau\lambda^{-q}R_q}{2\alpha-\mu}\right\}\left\{1-\frac{512\tau\lambda^{-q}R_q}{2\alpha-\mu}\right\}^{-1},\label{lem-bound-kappa}\\
&\xi:= 2\tau\left\{\frac{2\alpha-\mu}{8v}+\frac{1024\tau\lambda^{-q}R_q}{2\alpha-\mu}+5\right\}\left\{1-\frac{512\tau\lambda^{-q}R_q}{2\alpha-\mu}\right\}^{-1}.\label{lem-bound-xi}
\end{align}

For a given number $\delta>0$ and an integer $T>0$ such that
\begin{equation}\label{lem-cone-De1}
\obj(\Theta^t)-\obj(\hat{\Theta})\leq \delta, \quad \forall\ t\geq T,
\end{equation}
define
\begin{equation}\label{eq-epdel}
\epsilon(\delta):=2\min\left\{\frac{\delta}{\lambda},\omega\right\}.
\end{equation}

With this setup, we now state our main result on computational guarantee as follows.
\begin{Theorem}\label{thm-algo}
Let $R_q>0$ and $\omega>0$ be positive numbers such that $\Theta^*\in \B_q(R_q)\cap \Su$. Let $\hat{\Theta}$ be a global solution of the optimization problem \eqref{eq-esti}. Suppose that the nonconvex regularizer $\regu_\lambda$ satisfies Assumption \ref{asup-regu}, and that the empirical loss function $\loss_N$ satisfies the \emph{RSC} and \emph{RSM} conditions (cf. \eqref{eq-alg-rsc} and \eqref{eq-alg-rsm})  with $\alpha_2>\mu/2$.
Assume that $(\lambda, \omega)$ are chosen to satisfy 
\begin{equation}\label{eq-lambda-alg}
\lambda\geq  \max\left\{4\normm{\nabla\loss_N(\Theta^*)}_\emph{op}, 8\omega\tau,\left(\frac{256\tau R_q}{2\alpha_2-\mu}\right)^{1/q}\right\}.
\end{equation}
Let $\{\Theta^t\}_{t=0}^\infty$ be a sequence of iterates generated via \eqref{algo-pga} with an initial point $\Theta^0$ and $v\geq \max\{\frac{2\alpha_2-\mu}{4}, 2\alpha_3\}$. 
Then for any tolerance $\delta^*\geq\frac{8\xi}{1-\kappa}\bar{\epsilon}_{\emph{stat}}^2$ and any iteration $t\geq T(\delta^*)$, we have that
\begin{equation}\label{thm2-error}
\normm{\Theta^t-\hat{\Theta}}_\text{F}^2\leq \frac{4}{2\alpha_2-\mu}\left(\delta^*+\frac{{\delta^*}^2}{8\tau\omega^2}+4\tau\bar{\epsilon}_{\emph{stat}}^2\right),
\end{equation}
where
\begin{equation}\label{eq-Tdelta}
\begin{aligned}
T(\delta^*)&:=\log_2\log_2\left(\frac{\omega\lambda}{\delta^*}\right)\left(1+\frac{\log 2}{\log(1/\kappa)}\right) +\frac{\log((\obj(\Theta^0)-\obj(\hat{\Theta}))/\delta^*)}{\log(1/\kappa)},
\end{aligned}
\end{equation}
and $\bar{\epsilon}_{\emph{stat}}$, $\kappa$, $\xi$ are defined in
\eqref{bar-epsilon}-\eqref{lem-bound-xi}, respectively.
\end{Theorem}

\begin{Remark}\label{rmk-algo}
{\rm (i)} Due to the nonconvexity of \eqref{eq-esti}, it is usually difficult to obtain a global solution in general. However, Theorem \ref{thm-algo} shows that this is not a question under suitable conditions. Specifically, Theorem \ref{thm-algo} establishes the upper bound on the squared Frobenius norm between $\Theta^t$ at iteration $t$ and any global solution $\hat{\Theta}$. The iteration sequence can be easily computed via the proposed proximal gradient method, implying that the nonconvex estimator is solvable in practice. Moreover, Theorem \ref{thm-algo} guarantees that the proximal gradient method converges geometrically to a small neighborhood of all global optima. Noting from \eqref{thm2-error} that the optimization error $\normm{\Theta^t-\hat{\Theta}}_\emph{F}$ depends on the statistical error $\normm{\hat{\Theta}-\Theta^*}_\emph{F}$ ignoring constants, one sees that the converged point essentially behaves as well as any global solution of the nonconvex problem \eqref{eq-esti}, in the sense of statistical error. Hence, Theorem \ref{thm-algo} provides some prospect on nonconvex optimization problems that though global optima cannot be achieved generally, a near-global solution can still be computed via certain numerical algorithms and works as a good candidate under suitable regularity conditions.

{\rm (ii)} The geometric convergence rate is revealed by the number of required iterations $T(\delta^*)$ (cf. \eqref{eq-Tdelta}) in a logarithmic scale. It is worth noting that the convergence is not guaranteed to an arbitrary precision, but to an accuracy dependent on the parameters $\bar{\epsilon}_{\emph{stat}}$ (cf. \eqref{bar-epsilon}), $\kappa$ (cf. \eqref{lem-bound-kappa}) and $\xi$ (cf. \eqref{lem-bound-xi}). The quantity $\bar{\epsilon}_{\emph{stat}}$ (cf. \eqref{bar-epsilon}) consists of the statistical error $\normm{\hat{\Theta}-\Theta^*}_\emph{F}$ with some constants and the term $\lambda^{1-q}R_q$, and the latter one appears for the correction of the near low-rank assumption; see Remark \eqref{rmk-stat} (iv) for a detailed discussion. It has been pointed out in \cite{agarwal2012fast} that when solving optimization problems in statistical settings, there is no point to pursue a computational precision higher than the statistical precision. Therefore, the converged near-global solution is the best one we can expect from the point of statistical computing. 

{\rm (iii)} Theorem \ref{thm-algo} establishes the geometric convergence rate when $\Theta^*\in \B_q(R_q)$ with  $q\in [0,1]$ and suggests some important differences between the exact low-rank and the near low-rank assumptions, a finding which has also been discovered in \cite{agarwal2012fast,li2020sparse,Li2024LowrankME}. Concretely, under the exact low-rank assumption (i.e., $q=0$), the parameter $\bar{\epsilon}_\emph{stat}$ (cf. \eqref{bar-epsilon}) actually does not involve the second term $\lambda^{1-q}R_q$ (To see this, turn to the proof of Lemma \ref{lem-cone} (cf. inequality (\eqref{lem-cone-9})) and note the fact that $\sum_{j=r+1}^d\sigma_j(\Theta^*)=0$. This case was not explicitly shown in order to make the compactness of the article.)  Thus in the exact low-rank case, the optimization error $\normm{\Theta^t-\hat{\Theta}}_\emph{F}$ only depends on the statistical error $\normm{\hat{\Theta}-\Theta^*}_\emph{F}$. In contrast, under the near low-rank assumption, the squared optimization error \eqref{thm2-error} also has an additional term $\lambda^{2-2q}R_q^2$ besides the squared statistical error $\normm{\hat{\Theta}-\Theta^*}_\emph{F}^2$ (To see this, take the square of \eqref{bar-epsilon} with some computation.) This additional term arises from the essence of the matrix $\ell_q$-ball (i.e., the statistical nonidentifiability), and it is no larger than $\normm{\hat{\Theta}-\Theta^*}_\text{F}^2$ with overwhelming probability.
\end{Remark}

Before the proof Theorem \ref{thm-algo}, we need two technical lemmas as follows.

\begin{Lemma}\label{lem-cone}
Suppose that the conditions of Theorem \ref{thm-algo} are satisfied, and that there exists a pair $(\delta, T)$ such that \eqref{lem-cone-De1} holds.
Then for any iteration $t\geq T$, it holds that
\begin{equation*}
\begin{aligned}
\normm{\Theta^t-\hat{\Theta}}_*&\leq 4\sqrt{2}\lambda^{-\frac{q}{2}}R_q^{\frac12}\normm{\Theta^t-\hat{\Theta}}_\text{F}+\bar{\epsilon}_{\emph{stat}}+\epsilon(\delta),
\end{aligned}
\end{equation*}
where $\bar{\epsilon}_{\emph{stat}}$ and $\epsilon(\delta)$ are defined respectively in \eqref{bar-epsilon} and \eqref{eq-epdel}.
\end{Lemma}
\begin{proof}
We first prove that if $\lambda\geq 4\normm{\nabla\loss_N(\Theta^*)}_\text{op}$, then for any $\Theta\in \Su$ satisfying
\begin{equation}\label{lem-cone-De2}
\obj(\Theta)-\obj(\Theta^*)\leq \delta,
\end{equation}
it holds that
\begin{equation}\label{lem-cone-1}
\begin{aligned}
\normm{\Theta-\Theta^*}_*&\leq 4\sqrt{2}\lambda^{-\frac{q}{2}}R_q^{\frac12}\normm{\Theta-\Theta^*}_\text{F}+4\lambda^{1-q}R_q+2\min\left\{\frac{\delta}{\lambda},\omega\right\}.
\end{aligned}
\end{equation}
Set $\Delta:=\Theta-\Theta^*$. From \eqref{lem-cone-De2}, one has that
\begin{equation*}
\loss_N(\Theta^*+\Delta)+\regu_\lambda(\Theta^*+\Delta)\leq \loss_N(\Theta^*)+\regu_\lambda(\Theta^*)+\delta.
\end{equation*}
Then subtracting $\inm{\nabla\loss_N(\Theta^*)}{\Delta}$ from both sides of the former inequality, we obtain that
\begin{equation}\label{lem-cone-2}
\T(\Theta^*+\Delta,\Theta^*)+\regu_\lambda(\Theta^*+\Delta)-\regu_\lambda(\Theta^*)\leq -\inm{\nabla\loss_N(\Theta^*)}{\Delta}+\delta.
\end{equation}
We now claim that
\begin{equation}\label{lem-cone-claim}
\regu_\lambda(\Theta^*+\Delta)-\regu_\lambda(\Theta^*)\leq \frac{\lambda}{2}\normm{\Delta}_*+\delta.
\end{equation}
In fact, combining \eqref{lem-cone-2} with the RSC condition \eqref{eq-alg-rsc} and H{\"o}lder's inequality, one has that
\begin{equation*}
\alpha_2\normm{\Delta}_\text{F}^2-\tau_2\normm{\Delta}_*^2+\regu_\lambda(\Theta^*+\Delta)-\regu_\lambda(\Theta^*)\leq \normm{\nabla\loss_N(\Theta^*)}_{\text{op}}\normm{\Delta}_*+\delta.
\end{equation*}
This inequality, together with the assumption that $\lambda\geq 4\normm{\nabla\loss_N(\Theta^*)}_{\text{op}}$, implies that
\begin{equation*}
\alpha_2\normm{\Delta}_\text{F}^2-\tau_2\normm{\Delta}_*^2+\regu_\lambda(\Theta^*+\Delta)-\regu_\lambda(\Theta^*)\leq \frac{\lambda}{4}\normm{\Delta}_*+\delta.
\end{equation*}
Noting the facts that $\alpha_2>0$ and that $\normm{\Delta}_*\leq \normm{\Theta^*}_*+\normm{\Theta}_*\leq 2\omega$, one arrives at \eqref{lem-cone-claim} by combining the assumption $\lambda\geq 8\omega\tau_2$. Let $J$ denote the index set corresponding to the $2r$ largest singular values of $\Delta$ with $r$ to be chosen later and recall the definitions given in \eqref{eq-thetaj}. It then follows from Lemma \ref{lem-decom}(ii) that 
$\regu_\lambda(\Theta^*)-\regu_\lambda(\Theta)\leq \lambda(\normm{\Delta_J}_*-\normm{\Delta_{J^c}}_*)+2\lambda\sum_{j=r+1}^d\sigma_j(\Theta^*)$. Combining this inequality with \eqref{lem-cone-claim}, one has that $0\leq \frac{3\lambda}{2}\normm{\Delta_J}_*-\frac{\lambda}{2}\normm{\Delta_{J^c}}_*+2\lambda \sum_{j=r+1}^{d}\sigma_j(\Theta^*)+\delta$, and consequently, $\normm{\Delta_{J^c}}_*\leq 3\normm{\Delta_J}_*+4\sum_{j=r+1}^{d}\sigma_j(\Theta^*)+\frac{2\delta}{\lambda}$. Using the trivial bound $\normm{\Delta}_*\leq
2\omega$, one has that
\begin{equation}\label{lem-cone-9}
\normm{\Delta}_*\leq 4\sqrt{2r}\normm{\Delta}_\text{F}+4\sum_{j=r+1}^d\sigma_j(\Theta^*)+2\min\left\{\frac{\delta}{\lambda},\omega\right\}.
\end{equation}
Recall the set $K_\eta$ defined in \eqref{eq-thresh} and set $r=|K_\eta|$. Combining \eqref{lem-cone-9} with \eqref{eq-thresh-bound} and setting $\eta=\lambda$, we arrive at \eqref{lem-cone-1}. We now verify that \eqref{lem-cone-De2} is satiafied by the matrix $\hat{\Theta}$ and $\Theta^t$, respectively. Since $\hat{\Theta}$ is the optimal solution, it holds that $\obj(\hat{\Theta})\leq \obj(\Theta^*)$, and by assumption \eqref{lem-cone-De1}, it holds that $\obj(\Theta^t)\leq \obj(\hat{\Theta})+\delta\leq \obj(\Theta^*)+\delta$. Consequently, it follows from \eqref{lem-cone-1} that
\begin{equation*}
\begin{aligned}
\normm{\hat{\Theta}-\Theta^*}_*&\leq 4\sqrt{2}\lambda^{-\frac{q}{2}}R_q^{\frac12}\normm{\hat{\Theta}-\Theta^*}_\text{F}+4\lambda^{1-q}R_q,\\
\normm{\Theta^t-\Theta^*}_*&\leq 4\sqrt{2}\lambda^{-\frac{q}{2}}R_q^{\frac12}\normm{\Theta^t-\Theta^*}_\text{F}+4\lambda^{1-q}R_q+2\min\left\{\frac{\delta}{\lambda},\omega\right\}.
\end{aligned}
\end{equation*}
By the triangle inequality, we then arrive at that
\begin{equation*}
\begin{aligned}
\normm{\Theta^t-\hat{\Theta}}_*
&\leq \normm{\hat{\Theta}-\Theta^*}_*+\normm{\Theta^t-\Theta^*}_*\\
&\leq 4\sqrt{2}\lambda^{-\frac{q}{2}}R_q^{\frac12}\left(\normm{\hat{\Theta}-\Theta^*}_\text{F}+\normm{\Theta^t-\Theta^*}_\text{F}\right)+8\lambda^{1-q}R_q+2\min\left\{\frac{\delta}{\lambda},\omega\right\}\\
&\leq 4\sqrt{2}\lambda^{-\frac{q}{2}}R_q^{\frac12}\normm{\Theta^t-\hat{\Theta}}_\text{F}+\bar{\epsilon}_{\text{stat}}+\epsilon(\delta).
\end{aligned}
\end{equation*}
The proof is complete.
\end{proof}

\begin{Lemma}\label{lem-Tphi-bound}
Suppose that the conditions of Theorem \ref{thm-algo} are satisfied, and that there exists a pair $(\delta, T)$ such that \eqref{lem-cone-De1} holds.
Then for any iteration $t\geq T$, we have that
\begin{align}
&\bar{\T}(\hat{\Theta},\Theta^t)\geq -2\tau(\bar{\epsilon}_{\emph{stat}}+\epsilon(\delta))^2,\label{lem-bound-T1}\\
&\obj(\Theta^t)-\obj(\hat{\Theta})\geq \frac{2\alpha_2-\mu}{4}\normm{\Theta^t-\hat{\Theta}}_\text{F}^2-2\tau(\bar{\epsilon}_{\emph{stat}}+\epsilon(\delta))^2, \label{lem-bound-T2}\\
&\obj(\Theta^t)-\obj(\hat{\Theta})\leq \kappa^{t-T}(\obj(\Theta^T)-\obj(\hat{\Theta}))+\frac{2\xi}{1-\kappa}(\bar{\epsilon}_{\emph{stat}}^2+\epsilon^2(\delta)),\label{lem-bound-0}
\end{align}
where $\bar{\epsilon}_{\emph{stat}}$, $\epsilon(\delta)$, $\kappa$ and $\xi$ are defined in
\eqref{bar-epsilon}, \eqref{eq-epdel}, \eqref{lem-bound-kappa} and \eqref{lem-bound-xi}, respectively.
\end{Lemma}
\begin{proof}
It follows from the RSC condition \eqref{eq-alg-rsc} and Lemma \ref{lem-qlambda} (cf. \eqref{lem-qlambda-13}) that 
\begin{equation}\label{lem-bound-1}
\bar{\T}(\hat{\Theta},\Theta^t)\geq \left(\alpha_2-\frac{\mu}{2}\right)\normm{\hat{\Theta}-\Theta^t}_\text{F}^2-\tau_2\normm{\hat{\Theta}-\Theta^t}_*^2.
\end{equation}
Combining Lemma \ref{lem-cone} and the assumption that
$\lambda\geq \left(\frac{256\tau R_q}{2\alpha_2-\mu}\right)^{1/q}$ yields that
\begin{equation*}
\bar{\T}(\hat{\Theta},\Theta^t)
\geq -2\tau(\bar{\epsilon}_{\text{stat}}+\epsilon(\delta))^2,
\end{equation*}
which establishes \eqref{lem-bound-T1}. Furthermore, by the convexity of $\normm{\cdot}_*$, one has that
\begin{equation}\label{lem-bound-2}
\lambda\normm{\Theta^t}_*-\lambda\normm{\hat{\Theta}}_*-\lambda\inm{\hat{G}}{\Theta^t-\hat{\Theta}} \geq 0,
\end{equation}
where $\hat{G}\in \partial\normm{\hat{\Theta}}_*$. 
It follows from the first-order optimality condition for $\hat{\Theta}$ that
\begin{equation}\label{lem-bound-3}
\inm{\nabla\obj(\hat{\Theta})}{\Theta^t-\hat{\Theta}}\geq 0.
\end{equation}
Combining \eqref{lem-bound-1}, \eqref{lem-bound-2} and \eqref{lem-bound-3}, we obtain that
\begin{equation*}
\obj(\Theta^t)-\obj(\hat{\Theta})\geq \left(\alpha_2-\frac{\mu}{2}\right)\normm{\Theta^t-\hat{\Theta}}_\text{F}^2-\tau_2\normm{\Theta^t-\hat{\Theta}}_*^2.
\end{equation*}
Then applying Lemma \ref{lem-cone} to bound the term $\normm{\Theta^t-\hat{\Theta}}_*^2$ and noting the assumption that $\lambda\geq \left(\frac{256\tau R_q}{2\alpha_2-\mu}\right)^{1/q}$,
we arrive at \eqref{lem-bound-T2}.
Now we turn to prove \eqref{lem-bound-0}. Define
\begin{equation*}
\obj_t(\Theta):=\bar{\loss}_N(\Theta^t)+\inm{\nabla{\bar{\loss}_N(\Theta^t)}}{\Theta-\Theta^t}+\frac{v}{2}\normm{\Theta-\Theta^t}_\text{F}^2+\lambda\normm{\Theta}_*,
\end{equation*}
which is the objective function minimized over the feasible region $\Su=\{\Theta\big|\normm{\Theta}_*\leq \omega\}$ at iteration $t$. For any $a\in [0,1]$, it is easy to see that the matrix $\Theta_a=a\hat{\Theta}+(1-a)\Theta^t$ belongs to $\Su$ due to the convexity of $\Su$. Since $\Theta^{t+1}$ is the optimal solution of the optimization problem \eqref{algo-pga}, one has that
\begin{equation*}
\begin{aligned}
\obj_t(\Theta^{t+1}) &\leq \obj_t(\Theta_a)=\bar{\loss}_N(\Theta^t)+\inm{\nabla{\bar{\loss}_N(\Theta^t)}}{\Theta_a-\Theta^t} +\frac{v}{2}\|\Theta_a-\Theta^t\|_2^2+\lambda\normm{\Theta_a}_*\\
&\leq \bar{\loss}_N(\Theta^t)+\inm{\nabla\bar{\loss}_N(\Theta^t)}{a\hat{\Theta}-a\Theta^t} +\frac{va^2}{2}\normm{\hat{\Theta}-\Theta^t}_\text{F}^2+a\lambda\normm{\hat{\Theta}}_*+(1-a)\lambda\normm{\Theta^t}_*,
\end{aligned}
\end{equation*}
where the last inequality is from the convexity of $\normm{\cdot}_*$.
Then by \eqref{lem-bound-T1}, one sees that
\begin{equation}\label{lem-bound-4}
\begin{aligned}
\obj_t(\Theta^{t+1})
&\leq (1-a)\bar{\loss}(\Theta^t)+a\bar{\loss}(\hat{\Theta})+2a\tau(\bar{\epsilon}_{\text{stat}}+\epsilon(\delta))^2\\
&\quad +\frac{va^2}{2}\normm{\hat{\Theta}-\Theta^t}_\text{F}^2+a\lambda\normm{\hat{\Theta}}_*+(1-a)\lambda\normm{\Theta^t}_*\\
&\leq \obj(\Theta^t)-a(\obj(\Theta^t)-\obj(\hat{\Theta}))+2\tau(\bar{\epsilon}_{\text{stat}}+\epsilon(\delta))^2+\frac{va^2}{2}\normm{\hat{\Theta}-\Theta^t}_\text{F}^2.
\end{aligned}
\end{equation}
It then follows from the RSM condition \eqref{eq-alg-rsm} and Lemma \ref{lem-qlambda} (cf. \eqref{lem-qlambda-14}) that
\begin{equation*}
\begin{aligned}
\bar{\T}(\Theta^{t+1},\Theta^t)&\leq \alpha_3\normm{\Theta^{t+1}-\Theta^t}_\text{F}^2+\tau_3\normm{\Theta^{t+1}-\Theta^t}_*^2\\
&\leq \frac{v}{2}\normm{\Theta^{t+1}-\Theta^t}_\text{F}^2+\tau\normm{\Theta^{t+1}-\Theta^t}_*^2,
\end{aligned}
\end{equation*}
where the second inequality is due to assumption $v\geq 2\alpha_3$. Adding $\lambda\normm{\Theta^{t+1}}_*$ to both sides of the former inequality, we obtain that
\begin{equation*}
\begin{aligned}
\obj(\Theta^{t+1})
&\leq \bar{\loss}(\Theta^t)+\inm{\nabla\bar{\loss}(\Theta^t)}{\Theta^{t+1}-\Theta^t} +\lambda\normm{\Theta^{t+1}}_*+\frac{v}{2}\normm{\Theta^{t+1}-\Theta^t}_\text{F}^2+\tau\normm{\Theta^{t+1}-\Theta^t}_*^2\\
&=
\obj_t(\Theta^{t+1})+\tau\normm{\Theta^{t+1}-\Theta^t}_*^2.
\end{aligned}
\end{equation*}
Combining this inequality with \eqref{lem-bound-4} yields that
\begin{equation}\label{lem-bound-5}
\obj(\Theta^{t+1})\leq \obj(\Theta^t)-a(\obj(\Theta^t)-\obj(\hat{\Theta}))+\frac{va^2}{2}\normm{\hat{\Theta}-\Theta^t}_\text{F}^2+\tau\normm{\Theta^{t+1}-\Theta^t}_*^2+2\tau(\bar{\epsilon}_{\text{stat}}+\epsilon(\delta))^2.
\end{equation}
Define $\Delta^t:=\Theta^t-\hat{\Theta}$. Then it follows directly that
$\normm{\Theta^{t+1}-\Theta^t}_*^2\leq (\normm{\Delta^{t+1}}_*+\normm{\Delta^t}_*)^2\leq 2\normm{\Delta^{t+1}}_*^2+2\normm{\Delta^t}_*^2$. Combining this inequality with \eqref{lem-bound-5}, one has that
\begin{equation*}
\obj(\Theta^{t+1})\leq \obj(\Theta^t)-a(\obj(\Theta^t)-\obj(\hat{\Theta}))+\frac{va^2}{2}\normm{\hat{\Theta}-\Theta^t}_\text{F}^2+2\tau(\normm{\Delta^{t+1}}_*^2+\normm{\Delta^t}_*^2)+2\tau(\bar{\epsilon}_{\text{stat}}+\epsilon(\delta))^2.
\end{equation*}
To simplify the notations, define $\psi:=\tau(\bar{\epsilon}_{\text{stat}}+\epsilon(\Delta))^2$, $\zeta:=\tau\lambda^{-q}R_q$ and $\delta_t:=\obj(\Theta^t)-\obj(\hat{\Theta})$. Using Lemma \ref{lem-cone} to bound the term $\normm{\Delta^{t+1}}_*^2$ and $\normm{\Delta^t}_*^2$, we arrive at that
\begin{equation}\label{lem-bound-6}
\begin{aligned}
\obj(\Theta^{t+1})
&\leq \obj(\Theta^t)-a(\obj(\Theta^t)-\obj(\hat{\Theta}))+\frac{va^2}{2}\normm{\Delta^t}_\text{F}^2+128\tau\lambda^{-q}R_q(\normm{\Delta^{t+1}}_\text{F}^2+\normm{\Delta^t}_\text{F}^2)+10\psi\\
&= \obj(\Theta^t)-a(\obj(\Theta^t)-\obj(\hat{\Theta}))+\left(\frac{va^2}{2}+128\zeta\right)\normm{\Delta^t}_\text{F}^2+128\zeta\normm{\Delta^{t+1}}_\text{F}^2+10\psi.
\end{aligned}
\end{equation}
Then subtracting $\obj(\hat{\Theta})$ from both sides of \eqref{lem-bound-6}, one has by \eqref{lem-bound-T2} that
\begin{equation*}
\begin{aligned}
\delta_{t+1}
&\leq (1-a)\delta_t+\frac{2va^2+512\zeta}{2\alpha_2-\mu}(\delta_t+2\psi) +\frac{512\zeta}{2\alpha_2-\mu}(\delta_{t+1}+2\psi)+10\psi.
\end{aligned}
\end{equation*}
Setting $a=\frac{2\alpha_2-\mu}{4v}\in (0,1)$, it follows from the former inequality that
\begin{equation*}
\begin{aligned}
\left(1-\frac{512\zeta}{2\alpha_2-\mu}\right)\delta_{t+1}&\leq \left(1-\frac{2\alpha_2-\mu}{8v}+\frac{512\zeta}{2\alpha_2-\mu}\right)\delta_t+2\left(\frac{2\alpha_2-\mu}{8v}+\frac{1024\zeta}{2\alpha_2-\mu}+5\right)\psi,
\end{aligned}
\end{equation*}
or equivalently, $\delta_{t+1}\leq \kappa\delta_t+\xi(\bar{\epsilon}_{\text{stat}}+\epsilon(\delta))^2$, where $\kappa$ and $\xi$ were previously defined in \eqref{lem-bound-kappa} and \eqref{lem-bound-xi}, respectively. Finally, it follows that
\begin{equation*}
\begin{aligned}
\delta_t &\leq \kappa^{t-T}\delta_T+\xi(\bar{\epsilon}_{\text{stat}}+\epsilon(\delta))^2(1+\kappa+\kappa^2+\cdots+\kappa^{t-T-1})\\
&\leq \kappa^{t-T}\delta_T+\frac{\xi}{1-\kappa}(\bar{\epsilon}_{\text{stat}}+\epsilon(\delta))^2\leq
\kappa^{t-T}\delta_T+\frac{2\xi}{1-\kappa}(\bar{\epsilon}_{\text{stat}}^2+\epsilon^2(\delta)).
\end{aligned}
\end{equation*}
The proof is complete.
\end{proof}

By virtue of the above lemmas, we are now ready to prove Theorem \ref{thm-algo}. 
\begin{proof}[Proof of Theorem \ref{thm-algo}]
We first prove the inequality as follows:
\begin{equation}\label{thm2-1}
\obj(\Theta^t)-\obj(\hat{\Theta})\leq \delta^*,\quad \forall t\geq T(\delta^*).
\end{equation}
First divide iterations $t=0,1,\cdots$ into a sequence of disjoint epochs $[T_k,T_{k+1}]$. Then define the associated sequence of tolerances $\delta_0>\delta_1>\cdots$ such that
\begin{equation*}
\obj(\Theta^t)-\obj(\hat{\Theta})\leq \delta_k,\quad \forall t\geq T_k,
\end{equation*}
and the corresponding error term $\epsilon_k:=2\min \left\{\frac{\delta_k}{\lambda},\omega\right\}$. The values of $\{(\delta_k,T_k)\}_{k\geq 1}$ will be determined later.
At the first iteration,
we use Lemma \ref{lem-Tphi-bound} (cf. \eqref{lem-bound-0}) with $\epsilon_0=2\omega$ and $T_0=0$
to obtain that
\begin{equation}\label{thm2-T0}
\obj(\Theta^t)-\obj(\hat{\Theta})\leq \kappa^t(\obj(\Theta^0)-\obj(\hat{\Theta}))+\frac{2\xi}{1-\kappa}(\bar{\epsilon}_{\text{stat}}^2 +4\omega^2),\quad \forall t\geq T_0.
\end{equation}
Set $\delta_1:=\frac{4\xi}{1-\kappa}(\bar{\epsilon}_{\text{stat}}^2 +4\omega^2)$. Noting that $\kappa\in (0,1)$ by assumption, it follows from \eqref{thm2-T0} that for $T_1:=\lceil \frac{\log (2\delta_0/\delta_1)}{\log (1/\kappa)} \rceil$,
\begin{equation*}
\begin{aligned}
\obj(\Theta^t)-\obj(\hat{\Theta})&\leq \frac{\delta_1}{2}+\frac{2\xi}{1-\kappa}\left(\bar{\epsilon}_{\text{stat}}^2 +4\omega^2\right)=\delta_1\leq \frac{8\xi}{1-\kappa}\max\left\{\bar{\epsilon}_{\text{stat}}^2,4\omega^2\right\},\quad \forall t\geq T_1.
\end{aligned}
\end{equation*}
For $k\geq 1$, define
\begin{equation}\label{thm2-DeltaT}
\delta_{k+1}:= \frac{4\xi}{1-\kappa}(\bar{\epsilon}_{\text{stat}}^2+\epsilon_k^2)
\quad \mbox{and}\quad
T_{k+1}:= \left\lceil \frac{\log (2\delta_k/\delta_{k+1})}{\log (1/\kappa)}+T_k \right\rceil.
\end{equation}
Then we use Lemma \ref{lem-Tphi-bound} (cf. \eqref{lem-bound-0}) to obtain that for all $t\geq T_k$,
\begin{equation*}
\obj(\Theta^t)-\obj(\hat{\Theta})\leq \kappa^{t-T_k}(\obj(\Theta^{T_k})-\obj(\hat{\Theta}))+\frac{2\xi}{1-\kappa}(\bar{\epsilon}_{\text{stat}}^2+\epsilon_k^2),
\end{equation*}
which further implies that
\begin{equation*}
\obj(\Theta^t)-\obj(\hat{\Theta})\leq \delta_{k+1}\leq \frac{8\xi}{1-\kappa}\max\{\bar{\epsilon}_{\text{stat}}^2,\epsilon_k^2\},\quad \forall t\geq T_{k+1}.
\end{equation*}
On the other hand, it follows from \eqref{thm2-DeltaT} that the recursion for $\{(\delta_k,T_k)\}_{k=0}^\infty$ is as follows
\begin{subequations}
\begin{align}
\delta_{k+1}&\leq \frac{8\xi}{1-\kappa}\max\{\epsilon_k^2,\bar{\epsilon}_{\text{stat}}^2\},\label{thm2-recur-Delta}\\
T_k&\leq k+\frac{\log (2^k\delta_0/\delta_k)}{\log (1/\kappa)}.\label{thm2-recur-T}
\end{align}
\end{subequations}
Then by \cite[Section 7.2]{agarwal2012supplementaryMF}, one has from \eqref{thm2-recur-Delta} that
\begin{equation}\label{thm2-recur2}
\delta_{k+1}\leq \frac{\delta_k}{4^{2^{k+1}}}\quad \mbox{and}\quad \frac{\delta_{k+1}}{\lambda}\leq \frac{\omega}{4^{2^k}},\quad \forall k\geq 1.
\end{equation}
Now let us show how to decide the smallest $k$ such that $\delta_k\leq \delta^*$ via \eqref{thm2-recur2}. If we are at the first epoch, \eqref{thm2-1} is clearly satisfied due to \eqref{thm2-recur-Delta}.
Otherwise, from \eqref{thm2-recur-T}, one sees that $\delta_k\leq \delta^*$ holds after at most $$k(\delta^*)\geq \frac{\log(\log(\omega\lambda/\delta^*)/\log 4)}{\log(2)}+1=\log_2\log_2(\omega\lambda/\delta^*)$$ epoches.
Combining the above bound on $k(\delta^*)$ with \eqref{thm2-recur-T}, we obtain
that $\obj(\Theta^t)-\obj(\hat{\Theta})\leq \delta^*$ holds for all iterations
\begin{equation*}
t\geq
\log_2\log_2\left(\frac{\omega\lambda}{\delta^*}\right)\left(1+\frac{\log 2}{\log(1/\kappa)}\right)+\frac{\log(\delta_0/\delta^*)}{\log(1/\kappa)},
\end{equation*}
from which \eqref{thm2-1} is established.
Finally, as \eqref{thm2-1} is proved, we have by \eqref{lem-bound-T2} in Lemma \ref{lem-Tphi-bound}
and assumption $\lambda\geq \left(\frac{256\tau R_q}{2\alpha_2-\mu}\right)^{1/q}$ that, for any $t\geq T(\delta^*)$,
\begin{equation*}
\left(\frac{2\alpha_2-\mu}{4}\right)\normm{\Theta^t-\hat{\Theta}}_\text{F}^2
\leq \obj(\Theta^t)-\obj(\hat{\Theta}) +2\tau\left(\epsilon(\delta^*)+\bar{\epsilon}_{\text{stat}}\right)^2
\leq \delta^*+2\tau\left(\frac{2\delta^*}{\lambda}+\bar{\epsilon}_{\text{stat}}\right)^2.
\end{equation*}
Hence, for any $t\geq T(\delta^*)$,
one has by \eqref{eq-lambda-alg} that
\begin{equation*}
\normm{\Theta^t-\hat{\Theta}}_\text{F}^2\leq \frac{4}{2\alpha_2-\mu}\left(\delta^*+\frac{{\delta^*}^2}{8\tau\omega^2}+4\tau\bar{\epsilon}_{\text{stat}}^2\right).
\end{equation*}
The proof is complete.
\end{proof}

\section{Errors-in-variables matrix regression}\label{sec-conse}


In this section, we investigate high-dimensional errors-in-variables matrix regression with nosiy observations as an application of the nonconvex regularized method. The matrix regression model has been first proposed in \cite{wainwright2014structured} and has been studied from both theoretical and applicational aspects in the last decade; see, e.g., \cite{alquier2013rank,candes2010power,negahban2011estimation,recht2010guaranteed} and references therein. Consider the following generic observation model
\begin{equation}\label{eq-matrix}
y_i=\langle\langle X_i,\Theta^* \rangle\rangle+\epsilon_i,\quad i=1,2,\cdots,N,
\end{equation}
where for two matrices $X_i, \Theta^*\in \R^{d_1\times d_2}$, the inner product is defined as $\langle\langle X_i,\Theta^* \rangle\rangle:=\text{trace}(\Theta^*X_i^\top)$, $\{y_i\}_{i=1}^N$ are observation values for the response variable, $\{X_i\}_{i=1}^N$ are covariate matrices, $\{\epsilon_i\}_{i=1}^N$ are observation noise, $\Theta^*$ is the unknown parameter matrix to be estimated. With operator-theoretic notation, model \eqref{eq-matrix} can be written more compactly as 
\begin{equation}\label{eq-matrix-op}
y=\X(\Theta^*)+\epsilon,
\end{equation}
where $y:=(y_1,y_2,\cdots,y_N)\top\in \R^N$ is the response vector, $\epsilon:=(\epsilon_1,\epsilon_2,\cdots,\epsilon_N)\top\in \R^N$, the operator $\X:\R^{d_1\times d_2}\to \R^N$ is defined by the covariate matrices $\{X_i\}_{i=1}^N$ as $[\X(\Theta^*)]_i=\langle\langle X_i,\Theta^* \rangle\rangle$.

The observations $\{(X_i,y_i)\}_{i=1}^N$ are assumed to be fully-observed in standard formulations. However, this assumption is not realistic in many applications, where the covariates may be observed with noise or miss values and one can only observe the pairs $\{(Z_i,y_i)\}_{i=1}^N$ instead, where $Z_i$'s are noisy observations of the corresponding true $X_i$'s. We mainly consider the following two types of errors:
\begin{enumerate}[(a)]
\item Additive noise: For each $i=1,2,\cdots,N$, one observes $Z_{i}=X_{i}+W_{i}$, where $W_{i}\in \R^{d_1\times d_2}$ is a random matrix.
\item Missing data: For each $i=1,2,\cdots,N$, one observes a random matrix $Z_{i}\in \R^{d_1\times d_2}$, such that for each $j=1,2\cdots,d_1$, $k=1,2\cdots,d_2$, $({Z_i})_{jk}=({X_i})_{jk}$ with probability $1-\rho$, and $({Z_i})_{jk}=*$ with probability $\rho$, where $\rho\in [0,1)$.
\end{enumerate}
Throughout this section, we impose a Gaussian-ensemble assumption on the errors-in-variables matrix regression model. Specifically, for a matrix $X\in \R^{d_1\times d_2}$, define $M=d_1\times d_2$ and let $\text{vec}(X)\in \R^M$ stand for the vectorized form of the matrix $X$. Then the $\Sigma$-ensemble is defined as follows.
\begin{Definition}\label{def-sig}
Let $\Sigma\in \R^{M\times M}$ be a symmetric positive definite matrix. Then a matrix $X$ is said to be observed form the $\Sigma$-ensemble if $\emph{vec}(X)\sim \N(0,\Sigma)$.
\end{Definition}
Then for $i=1,2,\cdots,N$, the true observed matrix $X_i$ and the noisy matrix $W_i$ are assumed to be drawn from the $\Sigma_x$-ensemble and the $\Sigma_w$-ensemble, repectively. Observation noise on the response variables, i.e., $\epsilon$, is assumed to obey the Gaussian distribution with $\epsilon\sim \N(0,\sigma_\epsilon^2\I_N)$.

In order to estimate the true parameter $\Theta^*$ in \ref{eq-matrix} with clean covariate data, it is natural to introduce the empirical loss function as 
\begin{equation}\label{eq-vec}
\loss_N(\Theta)=\frac{1}{2N}\|y-\X(\Theta)\|_2^2=\frac{1}{2N}\sum_{i=1}^N(y_i-\text{vec}(X_i)^\top\text{vec}(\Theta))^2.
\end{equation}
Then define $\XX=(\text{vec}(X_1),\text{vec}(X_2),\cdots,\text{vec}(X_N))^\top\in \R^{N\times M}$, and one sees that \eqref{eq-vec} can be written as
\begin{equation}\label{eq-vec}
\loss_N(\Theta)=\frac{1}{2N}\|y-\XX \text{vec}(\Theta)\|_2^2=\frac{1}{2N}\|y\|_2^2-\frac{1}{N}\langle \XX^\top y, \text{vec}(\Theta)\rangle+\frac{1}{2N}\langle \XX^\top \XX \text{vec}(\Theta), \text{vec}(\Theta)\rangle. 
\end{equation}
However, in the errors-in-variables case, the quantities $\frac{\XX^\top \XX}{N}$
and $\frac{\XX^\top y}{N}$ in \eqref{eq-vec} are both unknown, meaning that this loss function can not work. Nevertheless, this transformation still offers some heuristic to construct the loss function in the errors-in-variables case through the plug-in method proposed in \cite{loh2012high}. Specifically, given a set of observations, one way is to find suitable estimates of the quantities $\frac{\XX^\top \XX}{N}$ and $\frac{\XX^\top Y}{N}$ that are adapted to the cases of additive noise and/or missing data. 

Let $(\hat{\Gamma},\hat{\Upsilon})$ be estimators of $(\frac{\XX^\top \XX}{N},\frac{\XX^\top Y}{N})$. Then recalling the nonconvex estimation method proposed in \ref{eq-esti}, we obtain the following estimator in the errors-in-variables case
\begin{equation}\label{eq-esti-error}
\hat{\Theta} \in \argmin_{\Theta\in \Omega\subseteq \R^{d_1\times d_2}}\left\{\frac{1}{2}\langle \hat{\Gamma}\text{vec}(\Theta), \text{vec}(\Theta)\rangle-\langle \hat{\Upsilon}, \text{vec}(\Theta)\rangle+\regu_\lambda(\Theta)\right\},
\end{equation}

Define $\Z=(\text{vec}(Z_1),\text{vec}(Z_2),\cdots,\text{vec}(Z_N))^\top\in \R^{N\times M}$. For the specific additive noise and missing data cases, as discussed in \cite{loh2012high}, an unbiased choice of the pair $(\hat{\Gamma},\hat{\Upsilon})$ is given respectively by
\begin{align}
\hat{\Gamma}_{\text{add}}&:=\frac{\Z^\top \Z}{N}-\Sigma_w \quad \mbox{and}\quad \hat{\Upsilon}_{\text{add}}:=\frac{\Z^\top y}{N}, \label{sur-add}\\
\hat{\Gamma}_{\text{mis}}&:=\frac{\tilde{\Z}^\top \tilde{\Z}}{N}-\rho\cdot\text{diag}\left(\frac{\tilde{\Z}^\top \tilde{\Z}}{N}\right) \quad \mbox{and}\quad \hat{\Upsilon}_{\text{mis}}:=\frac{\tilde{\Z}^\top y}{N}\quad \left(\tilde{\Z}=\frac{\Z}{1-\rho}\right). \label{sur-mis}
\end{align}

Under the high-dimensional scenario $(N\ll M)$, the estimated matrices $\hat{\Gamma}_{\text{add}}$ and $\hat{\Gamma}_{\text{mis}}$ in \eqref{sur-add} and \eqref{sur-mis} are always negative definite; actually, both the matrices $\Z^{\top}\Z$ and $\tilde{\Z}^{\top} \tilde{\Z}$ have rank at most $N$, and then the positive definite matrices $\Sigma_{w}$ and $\rho\cdot \text{diag}\left(\frac{\tilde{\Z}^\top \tilde{\Z}}{N}\right)$ are subtracted to get the estimates $\hat{\Gamma}_{\text{add}}$ and $\hat{\Gamma}_{\text{mis}}$, respectively. Therefore, the above estimator \eqref{eq-esti-error} is based on solving a optimization problem consisting of a loss function and a regularizer which are both nonconvex and falls into the framework of this article.

In order to apply the statistical and computational results in the previous section, we need to verify that the RSC/RSM conditions \eqref{eq-sta-rsc}, \eqref{eq-alg-rsc} and \eqref{eq-alg-rsm} are satisfied by the errors-in-variables matrix regression model with high probability. In addition, noting from \eqref{eq-lambda-sta} and \eqref{eq-lambda-alg}, we shall also bound the term $\normm{\nabla \loss_N(\Theta^*)}_\text{op}$ to provide some insight in setting the regularization parameter. Specifically, Proposition \ref{prop-add} is for the additive noise case, while Proposition \ref{prop-mis} is for the missing data case.

To simplify the notations, we define
\begin{align}
\tau_{\text{add}}&:= \lambda_{\text{min}}(\Sigma_x)\max\left\{\frac{(\normm{\Sigma_x}_\text{op}^2+\normm{\Sigma_w}_\text{op}^2}{\lambda^2_{\text{min}}(\Sigma_x)},1\right\},\label{tau-add}\\
\tau_{\text{mis}}&:= \lambda_{\text{min}}(\Sigma_x)\max\left\{\frac{1}{(1-\rho)^4}\frac{\normm{\Sigma_x}_\text{op}^4}{\lambda^2_{\text{min}}(\Sigma_x)},1\right\},\label{tau-mis}\\
\varphi_{\text{add}}&:=(\normm{\Sigma_x}_\text{op}+\normm{\Sigma_x}_\text{op})(\normm{\Sigma_x}_\text{op}+\sigma_\epsilon)\normm{\Theta^*}_\text{F},\label{deviation-add}\\
\varphi_{\text{mis}}&:=\frac{\normm{\Sigma_x}_\text{op}}{1-\rho}\left(\frac{\normm{\Sigma_x}_\text{op}}{1-\rho}+\sigma_\epsilon\right)\normm{\Theta^*}_\text{F}.\label{deviation-mis}
\end{align}
Note that $\varphi_{\text{add}}$ and $\varphi_{\text{mis}}$ respectively appears as the surrogate error term for errors-in-variables matrix regression and plays a key role in bounding the quantity $\normm{\nabla\loss_N(\Theta^*)}_\text{op}$ in the requirement of the regularization parameter $\lambda$ (cf. \eqref{eq-lambda-sta} and \eqref{eq-lambda-alg}).

\begin{Proposition}\label{prop-add}
In the additive noise case, let $\tau_{\emph{add}}$ and $\varphi_{\emph{add}}$ be defined as in \eqref{tau-add} and \eqref{deviation-add}, respectively. Then the following conclusions are true:\\
\emph{(i)} there exist universal positive constants $(c_0,c_1,c_2)$ such that with probability at least $1-c_1\exp\left(-c_2N\min\left\{\frac{\lambda^2_{\emph{min}}(\Sigma_x)}{(\normm{\Sigma_x}_\text{op}^2+\normm{\Sigma_w}_\text{op}^2)^2},1\right\}\right)$, the \emph{RSC} condition \eqref{eq-sta-rsc} holds with parameters
\begin{equation}\label{add-staRSC-para}
\alpha_1=\frac{1}{2}\lambda_{\emph{min}}(\Sigma_x),\ \text{and}\ \tau_1=c_0\tau_{\emph{add}}\sqrt{d_1d_2}\frac{\log d_1+\log d_2}{N},
\end{equation}
the \emph{RSC} condition \eqref{eq-alg-rsc} holds with parameters
\begin{equation}\label{add-algRSC-para}
\alpha_2=\frac{1}{4}\lambda_{\emph{min}}(\Sigma_x),\ \text{and}\ \tau_2=c_0\tau_{\emph{add}}\sqrt{d_1d_2}\frac{\log d_1+\log d_2}{N},
\end{equation}
the \emph{RSM} condition \eqref{eq-alg-rsm} holds with parameters
\begin{equation}\label{add-algRSM-para}
\alpha_3=\frac{3}{4}\lambda_{\emph{max}}(\Sigma_x),\ \text{and}\ \tau_3=c_0\tau_{\emph{add}}\sqrt{d_1d_2}\frac{\log d_1+\log d_2}{N};
\end{equation}
\emph{(ii)} there exist universal positive constants $(c_3,c_4,c_5)$ such that
\begin{equation*}
\normm{\nabla\loss_N(\Theta^*)}_\emph{op}\leq c_3\varphi_{\emph{add}}\sqrt{\frac{\log d_1+\log d_2}{N}},
\end{equation*}
holds with probability at least $1-c_4\exp(-c_5d_1d_2)$.\end{Proposition}

\begin{proof}
\rm{(i)} Note that the matrices $X$ and $W$ are sub-Gaussian with parameters $(\Sigma_x,\normm{\Sigma_x}_\text{op}^2)$ and $(\Sigma_w,\normm{\Sigma_w}_\text{op}^2)$, respectively \cite{loh2012supplementaryMH}. Then \cite[Lemma 1]{loh2012high} is applicable to concluding that, there exist universal constants $(c_0,c_1,c_2)$ such that for any $\beta\in \R^M$
\begin{align}
\beta^\top\hat\Gamma_{\text{add}}\beta&\geq \frac{1}{2}\lambda_{\text{min}}(\Sigma_x)\|\beta\|_2^2-c_0\tau_{\text{add}}\frac{\log M}{N}\|\beta\|_1^2,\label{add-RSC1}\\
\beta^\top\hat\Gamma_{\text{add}}\beta&\leq \frac{3}{2}\lambda_{\text{max}}(\Sigma_x)\|\beta\|_2^2+c_0\tau_{\text{add}}\frac{\log M}{N}\|\beta\|_1^2.\label{add-RSM}
\end{align}
with probability at least $1-c_1\exp\left(-c_2N\min\left\{\frac{\lambda^2_{\text{min}}(\Sigma_x)}{(\normm{\Sigma_x}_\text{op}^2+\normm{\Sigma_w}_\text{op}^2)^2},1\right\}\right)$. On the other hand, for any $\Theta,\Theta'\in \R^{d_1\times d_2}$, note that $\T(\Theta,\Theta')=\frac{1}{2}(\text{vec}(\Theta)-\text{vec}(\Theta')^\top\hat\Gamma_{\text{add}}(\text{vec}(\Theta)-\text{vec}(\Theta'))$. Then with $\text{vec}(\Theta)-\text{vec}(\Theta')$ in place of $\beta$ and noting the facts that $\|\text{vec}(\Theta)-\text{vec}(\Theta')\|_2^2=\normm{\Theta-\Theta'}_\text{F}^2$ and that $\|\text{vec}(\Theta)-\text{vec}(\Theta')\|_1^2\leq \sqrt{M}\|\text{vec}(\Theta)-\text{vec}(\Theta')\|_2^2=\sqrt{M}\normm{\Theta-\Theta'}_\text{F}^2\leq \sqrt{M}\normm{\Theta-\Theta'}_*^2$, we obtain that the RSC condition \eqref{eq-alg-rsc} and the RSM condition \eqref{eq-alg-rsm} hold with high probability with parameters given by \eqref{add-algRSC-para} and \eqref{add-algRSM-para}, respectively. 

Then let $\Theta^*$ be the true underlying parameter (i.e., $\Theta^*$ satisfies \eqref{eq-matrix-op}). Fix $\Delta\in \R^{d_1\times d_2}$ such that $\Theta^*+\Delta \in \Su$. Since in the case of additive noise, it follows that $\inm{\nabla\loss_N(\Theta^*+\Delta)-\nabla\loss_N(\Theta^*)}{\Delta}=\text{vec}(\Delta)^\top\hat\Gamma_{\text{add}}\text{vec}(\Delta)$.
Now with $\text{vec}(\Delta)$ in palce of $\beta$ and noting the facts that $\|\text{vec}(\Delta)\|_2^2=\normm{\Delta}_\text{F}^2$ and that $\|\text{vec}(\Delta)\|_1^2\leq \sqrt{M}\|\text{vec}(\Delta)\|_1^2=\sqrt{M}\normm{\Delta}_\text{F}^2\leq \sqrt{M}\normm{\Delta}_*^2$, one has that \eqref{add-RSC1} implies that the RSC condition \eqref{eq-sta-rsc} hold with high probability with parameters given by \eqref{add-staRSC-para}.

\rm{(ii)} Since 
\begin{equation*}
\normm{\nabla\loss_N(\Theta^*)}_\text{op}=\|\hat{\Gamma}_{\text{add}}\text{vec}(\Theta^*)-\hat{\Upsilon}_{\text{add}}\|_\infty,
\end{equation*}
\cite[Lemma 2]{loh2012high} is applicable to concluding that, there exist universal positive constants $(c_3,c_4,c_5)$ such that
\begin{equation*}
\normm{\nabla\loss_n(\Theta^*)}_\text{op}\leq c_3\varphi_{\text{add}}\sqrt{\frac{\log M}{N}},
\end{equation*}
holds with probability at least $1-c_4\exp(-c_5M)$.

The proof is complete.
\end{proof}

The proof of Proposition \ref{prop-mis} is similar to that of Proposition \ref{prop-add} with the modification that \cite[Lemma 3 and Lemma 4]{loh2012high} is applied instead of \cite[Lemma 1 and Lemma 2]{loh2012high}, and so is omitted.

\begin{Proposition}\label{prop-mis}
In the missing data case,  let $\tau_{\emph{mis}}$ and $\varphi_{\text{mis}}$ be defined as in \eqref{tau-mis} and \eqref{deviation-mis}, respectively. Then the following conclusions are true:\\
\emph{(i)} there exist universal positive constants $(c_0,c_1,c_2)$ such that with probability at least $1-c_1\exp\left(-c_2N\min\left\{(1-\rho)^4\frac{\lambda^2_{\emph{min}}(\Sigma_x)}{\normm{\Sigma_x}_\text{op}^4},1\right\}\right)$, the \emph{RSC} condition \eqref{eq-sta-rsc} holds with parameters
\begin{equation*}
\alpha_1=\frac{1}{2}\lambda_{\emph{min}}(\Sigma_x),\ \text{and}\ \tau_1=c_0\tau_{\emph{mis}}\sqrt{d_1d_2}\frac{\log d_1+\log d_2}{N},
\end{equation*}
the \emph{RSC} condition \eqref{eq-alg-rsc} holds with parameters
\begin{equation*}
\alpha_2=\frac{1}{4}\lambda_{\emph{min}}(\Sigma_x),\ \text{and}\ \tau_2=c_0\tau_{\emph{mis}}\sqrt{d_1d_2}\frac{\log d_1+\log d_2}{N},
\end{equation*}
the \emph{RSM} condition \eqref{eq-alg-rsm} holds with parameters
\begin{equation*}
\alpha_3=\frac{3}{4}\lambda_{\emph{max}}(\Sigma_x),\ \text{and}\ \tau_3=c_0\tau_{\emph{mis}}\sqrt{d_1d_2}\frac{\log d_1+\log d_2}{N};
\end{equation*}
\emph{(ii)} there exist universal positive constants $(c_3,c_4,c_5)$ such that
\begin{equation*}
\normm{\nabla\loss_N(\Theta^*)}_\emph{op}\leq c_3\varphi_{\emph{mis}}\sqrt{\frac{\log d_1+\log d_2}{N}},
\end{equation*}
holds with probability at least $1-c_4\exp(-c_5d_1d_2)$.
\end{Proposition}

\section{Numerical experiments}\label{sec-num}

\section{Conclusion}\label{sec-con}

In this work, we proposed a unified framework for low-rank estimation problem, which can be used to analyse high-dimensional errors-in-variables matrix regression. A general estimator consisting a loss function and a low-rank induced regularizer which are both nonconvex was constructed. Then we provided statistical and computational guarantees for the nonconvex estimator. In the statistical aspect, recovery bounds for any stationary point are established to achieve the statistical consistency. In the computational aspect, the proximal gradient algorithm was applied to solve the nonconvex  optimization problem and was proved to converge geometrically to a near-global solution. Probabilistic discussions on the required regularity conditions were obtained for specific errors-in-variables matrix models. Theoretical results were illustrated by several numerical experiments on errors-in-variables matrix regression models with additive noise and missing data.


\bibliographystyle{plain}
\bibliography{./matrixEIV-ref}

\begin{thebibliography}{10}

\bibitem{agarwal2012fast}
A.~Agarwal, S.~Negahban, and M.~J. Wainwright.
\newblock Fast global convergence of gradient methods for high-dimensional
  statistical recovery.
\newblock {\em Ann. Stat.}, 40(5):2452--2482, 2012.

\bibitem{agarwal2012supplementaryMF}
A.~Agarwal, S.~Negahban, and M.~J. Wainwright.
\newblock Supplementary {Material}: {Fast} global convergence of gradient
  methods for high-dimensional statistical recovery.
\newblock 2012.

\bibitem{alquier2013rank}
P.~Alquier, C.~Butucea, M.~Hebiri, K.~Meziani, and T.~Morimae.
\newblock Rank-penalized estimation of a quantum system.
\newblock {\em Phys. Rev. A}, 88(3):032113, 2013.

\bibitem{belloni2016an}
A.~Belloni, M.~Rosenbaum, and A.~B. Tsybakov.
\newblock An $\ell_1, \ell_2, \ell_\infty$-regularization approach to
  high-dimensional errors-in-variables models.
\newblock {\em Electron. J. Stat.}, 10(2):1729--1750, 2016.

\bibitem{belloni2017linear}
A.~Belloni, M.~Rosenbaum, and A.~B. Tsybakov.
\newblock Linear and conic programming estimators in high dimensional
  errors-in-variables models.
\newblock {\em J. Royal Stat. Soc. B}, 79(3):939--956, 2017.

\bibitem{bickel1987efficient}
P.~J. Bickel and Y.~Ritov.
\newblock Efficient estimation in the errors in variables model.
\newblock {\em Ann. Stat.}, 15(2):513--540, 1987.

\bibitem{brown2019meboost}
B.~Brown, T.~Weaver, and J.~Wolfson.
\newblock {MEBoost}: {Variable} selection in the presence of measurement error.
\newblock {\em Stat. Med.}, 38(15):2705--2718, 2019.

\bibitem{candes2007the}
E.~J. Cand{\`e}s and T.~Tao.
\newblock The {Dantzig} selector: {Statistical} estimation when $p$ is much
  larger than $n$.
\newblock {\em Ann. Stat.}, 35(6):2313--2351, 2007.

\bibitem{Cands2009ThePO}
E.~J. Cand{\`e}s and T~Tao.
\newblock The power of convex relaxation: {Near}-optimal matrix completion.
\newblock {\em IEEE Transactions on Information Theory}, 56:2053--2080, 2009.

\bibitem{candes2010power}
E.~J. Cand{\`e}s and T.~Tao.
\newblock The power of convex relaxation: {Near-optimal} matrix completion.
\newblock {\em IEEE Trans. Inf. Theory}, 56(5):2053--2080, 2010.

\bibitem{carroll2006measurement}
R.~J. Carroll, D.~Ruppert, C.~M. Crainiceanu, and L.~A. Stefanski.
\newblock {\em Measurement Error in Nonlinear Models: A Modern Perspective}.
\newblock CRC Press, Boca Raton, 2006.

\bibitem{chen2013noisy}
Y.~D. Chen and C.~Caramanis.
\newblock Noisy and missing data regression: {Distribution}-oblivious support
  recovery.
\newblock In {\em International Conference on Machine Learning}, pages
  383--391, 2013.

\bibitem{datta2017cocolasso}
A.~Datta and H.~Zou.
\newblock Cocolasso for high-dimensional error-in-variables regression.
\newblock {\em Ann. Stat.}, 45(6):2400--2426, 2017.

\bibitem{fan2001variable}
J.~Q. Fan and R.~Z. Li.
\newblock Variable selection via nonconcave penalized likelihood and its oracle
  properties.
\newblock {\em J. Am. Stat. Assoc.}, 96(456):1348--1360, 2001.

\bibitem{Fazel2001ARM}
M.~Fazel, H.~A. Hindi, and S.~P. Boyd.
\newblock A rank minimization heuristic with application to minimum order
  system approximation.
\newblock {\em Proceedings of the 2001 American Control Conference. (Cat.
  No.01CH37148)}, 6:4734--4739, 2001.

\bibitem{gui2015towards}
H.~Gui, J.~W. Han, and Q.~Q. Gu.
\newblock Towards faster rates and oracle property for low-rank matrix
  estimation.
\newblock In {\em Proc. ICML}, 2015.

\bibitem{li2021inference}
M.~Y. Li, R.~Z. Li, and Y.~Y. Ma.
\newblock Inference in high dimensional linear measurement error models.
\newblock {\em J. Multivar. Anal.}, page 104759, 2021.

\bibitem{li2023Lowrank}
X.~Li and D.~Y. Wu.
\newblock Low-rank matrix estimation via nonconvex optimization methods in
  multi-response errors-in-variables regression.
\newblock {\em J. Global Optim.}, 2023.

\bibitem{Li2024LowrankME}
X.~Li and D.~Y. Wu.
\newblock Low-rank matrix estimation via nonconvex optimization methods in
  multi-response errors-in-variables regression.
\newblock {\em J. Glob. Optim.}, 88:79--114, 2024.

\bibitem{li2020sparse}
X.~Li, D.~Y. Wu, C.~Li, J.~H. Wang, and J.-C. Yao.
\newblock Sparse recovery via nonconvex regularized {M}-estimators over
  $\ell_q$-balls.
\newblock {\em Comput. Stat. Data Anal.}, 152:107047, 2020.

\bibitem{loh2013local}
P.-L. Loh.
\newblock Local optima of nonconvex regularized {M}-estimators.
\newblock {Dept. Elect. Eng. Comput. Sci.}, UC Berkeley, Berkeley, 2013.

\bibitem{loh2012high}
P.-L. Loh and M.~J. Wainwright.
\newblock High-dimensional regression with noisy and missing data: Provable
  guarantees with nonconvexity.
\newblock {\em Ann. Stat.}, 40(3):1637--1664, 2012.

\bibitem{loh2012supplementaryMH}
P.-L. Loh and M.~J. Wainwright.
\newblock Supplementary material: {High}-dimensional regression with noisy and
  missing data: {Provable} guarantees with nonconvexity.
\newblock {\em Ann. Statist.}, 2012.

\bibitem{loh2015regularized}
P.-L. Loh and M.~J. Wainwright.
\newblock Regularized {M-estimators} with nonconvexity: {Statistical} and
  algorithmic theory for local optima.
\newblock {\em J. Mach. Learn. Res.}, 16(1):559--616, 2015.

\bibitem{Lu2014GeneralizedNN}
C.~Y. Lu, J.~H. Tang, S.~C. Yan, and Z.~C. Lin.
\newblock Generalized nonconvex nonsmooth low-rank minimization.
\newblock {\em 2014 IEEE Conference on Computer Vision and Pattern
  Recognition}, pages 4130--4137, 2014.

\bibitem{Mazumder2010SpectralRA}
R.~Mazumder, T.~J. Hastie, and R.~Tibshirani.
\newblock Spectral regularization algorithms for learning large incomplete
  matrices.
\newblock {\em Journal of machine learning research : JMLR}, 11:2287--2322,
  2010.

\bibitem{Natarajan1995SparseAS}
B.~K. Natarajan.
\newblock Sparse approximate solutions to linear systems.
\newblock {\em SIAM J. Comput.}, 24:227--234, 1995.

\bibitem{negahban2011estimation}
S.~Negahban and M.~J. Wainwright.
\newblock Estimation of (near) low-rank matrices with noise and
  high-dimensional scaling.
\newblock {\em Ann. Stat.}, 39(2):1069--1097, 2011.

\bibitem{nesterov2007gradient}
Y.~Nesterov.
\newblock Gradient methods for minimizing composite objective function.
\newblock Technical report, Universit{\'e} catholique de Louvain, Center for
  Operations Research and Econometrics (CORE), 2007.

\bibitem{nesterov2013introductory}
Y.~Nesterov.
\newblock {\em Introductory lectures on convex optimization: {A} basic course},
  volume~87, chapter~2, pages 56,61.
\newblock Springer Science \& Business Media, Berlin, 2013.

\bibitem{recht2010guaranteed}
B.~Recht, M.~Fazel, and P.~A. Parrilo.
\newblock Guaranteed minimum-rank solutions of linear matrix equations via
  nuclear norm minimization.
\newblock {\em SIAM Rev.}, 52(3):471--501, 2010.

\bibitem{rosenbaum2010sparse}
M.~Rosenbaum and A.~B. Tsybakov.
\newblock Sparse recovery under matrix uncertainty.
\newblock {\em Ann. Stat.}, 38(5):2620--2651, 2010.

\bibitem{rosenbaum2013improved}
M.~Rosenbaum and A.~B. Tsybakov.
\newblock Improved matrix uncertainty selector.
\newblock In {\em From Probability to Statistics and Back: High-Dimensional
  Models and Processes--A Festschrift in Honor of Jon A. Wellner}, pages
  276--290. Institute of Mathematical Statistics, 2013.

\bibitem{Rotfeld1967RemarksOT}
S.~Yu. Rotfel'd.
\newblock Remarks on the singular numbers of a sum of completely continuous
  operators.
\newblock {\em Functional Analysis and Its Applications}, 1:252--253, 1967.

\bibitem{sagan2021lowrank}
A.~Sagan and J.~E. Mitchell.
\newblock Low-rank factorization for rank minimization with nonconvex
  regularizers.
\newblock {\em Comput. Optim. Appl.}, 79:273--300, 2021.

\bibitem{sorensen2015measurement}
{\O}.~S{\o}rensen, A.~Frigessi, and M.~Thoresen.
\newblock Measurement error in {LASSO}: {Impact} and likelihood bias
  correction.
\newblock {\em Stat. Sinica}, pages 809--829, 2015.

\bibitem{sorensen2018covariate}
{\O}.~S{\o}rensen, K.~H. Hellton, A.~Frigessi, and M.~Thoresen.
\newblock Covariate selection in high-dimensional generalized linear models
  with measurement error.
\newblock {\em J. Comput. Graph. Stat.}, 27(4):739--749, 2018.

\bibitem{tibshirani1996regression}
R.~Tibshirani.
\newblock Regression shrinkage and selection via the {Lasso}.
\newblock {\em J. Royal Stat. Soc. B}, 58(1):267--288, 1996.

\bibitem{wainwright2014structured}
M.~J. Wainwright.
\newblock Structured regularizers for high-dimensional problems: {Statistical}
  and computational issues.
\newblock {\em Annu. Rev. Stat. Appl.}, 1:233--253, 2014.

\bibitem{wu2020scalable}
J.~Wu, Z.~M. Zheng, Y.~Li, and Y.~Zhang.
\newblock Scalable interpretable learning for multi-response error-in-variables
  regression.
\newblock {\em J. Multivar. Anal.}, page 104644, 2020.

\bibitem{wu2019joint}
X.~Wu, X.~Zhang, N.~Wang, and Y.~Cen.
\newblock Joint sparse and low-rank multi-task learning with extended
  multi-attribute profile for hyperspectral target detection.
\newblock {\em Remote. Sens.}, 11:150, 2019.

\bibitem{yao2017large}
Q.~M. Yao, J.~T. Kwok, T.~F. Wang, and T.-Y. Liu.
\newblock Large-scale low-rank matrix learning with nonconvex regularizers.
\newblock {\em {IEEE} Trans. Pattern Anal. Machine Intell.}, 41:2628--2643,
  2017.

\bibitem{yao2015fast}
Q.~M. Yao, J.~T. Kwok, and L.~W. Zhong.
\newblock Fast low-rank matrix learning with nonconvex regularization.
\newblock {\em {IEEE} ICDM}, pages 539--548, 2015.

\bibitem{Yue2016API}
M.-C. Yue and M.-C.~S. Anthony.
\newblock A perturbation inequality for concave functions of singular values
  and its applications in low-rank matrix recovery.
\newblock {\em Applied and Computational Harmonic Analysis}, 40:396--416, 2016.

\bibitem{Zhang2010}
C.-H. Zhang.
\newblock Nearly unbiased variable selection under minimax concave penalty.
\newblock {\em Ann. Stat.}, 38(2):894--942, 2010.

\bibitem{zhou2014regularized}
H.~Zhou and L.~X. Li.
\newblock Regularized matrix regression.
\newblock {\em J. Royal Stat. Soc. B}, 76(2):463--483, 2014.

\bibitem{zou2005regularization}
H.~Zou and T.~Hastie.
\newblock Regularization and variable selection via the elastic net.
\newblock {\em J. Royal Stat. Soc. B}, 67(2):301--320, 2005.

\end{thebibliography}

\end{document}